\documentclass[11pt,reqno]{amsart}%
\usepackage{graphicx}
\usepackage{amsmath, stmaryrd}
\usepackage{mathtools, amsfonts, amssymb, amsthm}
\usepackage{enumerate, url}
\usepackage[margin=1in]{geometry}
\usepackage[normalem]{ulem}
\usepackage{caption}
\usepackage{algorithm}
\usepackage{algpseudocode}
\usepackage{mleftright}
\usepackage{tikz-cd}
\usepackage[normalem]{ulem}
\usepackage{hyperref}

\usepackage{xcolor}
\hypersetup{
    colorlinks,
    linkcolor={red!50!black},
    citecolor={blue!50!black},
    urlcolor={blue!80!black}
}

\newcommand{\mathbbm}[1]{\text{\usefont{U}{bbm}{m}{n}#1}}

\newtheorem{theorem}{Theorem}[section]
\newtheorem{proposition}[theorem]{Proposition}
\newtheorem{proposition/definition}[theorem]{Proposition/Definition}

\newtheorem{corollary}[theorem]{Corollary}

\theoremstyle{definition}
\newtheorem{definition}[theorem]{Definition}

\newtheorem{example}[theorem]{Example}

\theoremstyle{remark}

\newcommand{\tp}{{\scriptscriptstyle\mathsf{T}}}
\newcommand{\lb}{\llbracket}
\newcommand{\rb}{\rrbracket}

\let\O\undefined
\let\S\undefined
\let\SO\undefined
\DeclareMathOperator{\Gr}{Gr}
\DeclareMathOperator{\V}{V}

\DeclareMathOperator{\Flag}{Flag}
\DeclareMathOperator{\GL}{GL}

\DeclareMathOperator{\O}{O}
\DeclareMathOperator{\S}{S}
\DeclareMathOperator{\SO}{SO}
\DeclareMathOperator{\tr}{tr}
\DeclareMathOperator{\Ad}{Ad}

\DeclareMathOperator{\im}{im}
\DeclareMathOperator{\rank}{rank}
\DeclareMathOperator{\diag}{diag}

\DeclareMathOperator{\spn}{span}
\DeclareMathOperator*{\argmax}{argmax}

\DeclareMathOperator*{\degree}{deg}

\DeclareFontFamily{U} {MnSymbolC}{}
\DeclareFontShape{U}{MnSymbolC}{m}{n}{
  <-6> MnSymbolC5
  <6-7> MnSymbolC6
  <7-8> MnSymbolC7
  <8-9> MnSymbolC8
  <9-10> MnSymbolC9
  <10-12> MnSymbolC10
  <12-> MnSymbolC12}{}
\DeclareFontShape{U}{MnSymbolC}{b}{n}{
  <-6> MnSymbolC-Bold5
  <6-7> MnSymbolC-Bold6
  <7-8> MnSymbolC-Bold7
  <8-9> MnSymbolC-Bold8
  <9-10> MnSymbolC-Bold9
  <10-12> MnSymbolC-Bold10
  <12-> MnSymbolC-Bold12}{}

\DeclareSymbolFont{MnSyC} {U} {MnSymbolC}{m}{n}

\DeclareMathSymbol{\minus}{\mathrel}{MnSyC}{16}
\DeclareMathSymbol{\plus}{\mathrel}{MnSyC}{20}

\newcommand{\pp}{{\plus\plus}}

\newcommand{\Rmnum}[1]{\expandafter\@slowromancap\Romannumeral #1@}

\begin{document}
\title{Optimization on flag manifolds}
\author[K.~Ye]{Ke~Ye}
\address{KLMM, Academy of Mathematics and Systems Science, Chinese Academy of Sciences,
Beijing 100190, China}
\email{keyk@amss.ac.cn}
\author[K.~S.-W.~Wong]{Ken Sze-Wai Wong}
\address{Department of Statistics,
University of Chicago, Chicago, IL 60637-1514.}
\email{kenwong@uchicago.edu}
\author[L.-H.~Lim]{Lek-Heng~Lim}
\address{Computational and Applied Mathematics Initiative, Department of Statistics,
University of Chicago, Chicago, IL 60637-1514.}
\email{lekheng@galton.uchicago.edu}

\begin{abstract}
A flag is a sequence of nested subspaces.  Flags are ubiquitous in numerical analysis, arising in finite elements, multigrid, spectral, and pseudospectral methods for numerical \textsc{pde}; they arise in the form of Krylov subspaces in matrix computations, and as multiresolution analysis in wavelets constructions. They are common in statistics too --- principal component, canonical correlation, and correspondence analyses may all be viewed as methods for extracting flags from a data set. The main goal of this article is to develop the tools needed for optimizing over a set of flags, which is a smooth manifold called the flag manifold, and it contains the Grassmannian as the simplest special case. We will derive closed-form analytic expressions for various differential geometric objects required for Riemannian optimization algorithms on the flag manifold; introducing various systems of extrinsic coordinates that allow us to parameterize points, metrics, tangent spaces, geodesics, distance, parallel transport, gradients, Hessians in terms of matrices and matrix operations; and thereby permitting us to formulate steepest descent, conjugate  gradient, and Newton algorithms on the flag manifold using only standard numerical linear algebra.

\end{abstract}

\subjclass[2010]{62H12, 14M15, 90C30, 62H10, 68T10}

\keywords{Flag manifold, 
generalized flag variety, 
linear subspaces, 
distances and metrics, 
manifold optimization, 
multivariate data analysis,
numerical analysis}

\maketitle

\section{Introduction}

Launched around 20 years ago in a classic article of Edelman, Arias, and Smith \cite{EAS},  Riemannian \emph{manifold optimization}  is now entrenched as a mainstay of optimization theory \cite{ABG2007,AMS2009,CDH2011,WY2013}. While studies of optimization algorithms on Riemannian manifolds predate \cite{EAS}, the distinguishing feature of  Edelman et al.'s approach is that their algorithms are built entirely and directly from standard algorithms in numerical linear algebra; in particular, they do not require numerical solutions of differential equations. For instance, the parallel transport of a vector in \cite{EAS} is not merely discussed in the abstract but may be explicitly computed in efficient and numerically stable ways via closed-form analytic expressions involving QR and singular value decompositions of various matrices.

The requirement that differential geometric quantities appearing in a manifold optimization algorithms have analytic expressions in terms of standard matrix decompositions limits the type of Riemannian manifolds that one may  consider. Aside from Euclidean spaces, we know of exactly three Riemannian manifolds \cite{AMS2009} on which one may define optimization algorithms in this manner:
\begin{enumerate}[\upshape (i)]
\item\label{it:V} Stiefel manifold $\V(k,n)$,
\item\label{it:Gr} Grassmann manifold $\Gr(k,n)$,
\item manifold of positive definite matrices $\mathbb{S}^n_{\pp}$.
\end{enumerate}
The main contribution of this article is to furnish a fourth: \emph{flag manifolds}.

A flag in a finite-dimensional vector space $\mathbb{V}$ over $\mathbb{R}$ is a nested sequence of linear subspaces $\{\mathbb{V}_i\}_{i=1}^d$ of $\mathbb{V}$, i.e.,
\[
\{0\} \subsetneq \mathbb{V}_1 \subsetneq \dots \subsetneq \mathbb{V}_d \subsetneq \mathbb{V}.
\]
For any increasing integer sequence of length $d$, $0 < n_1 < \dots < n_d < n$, the set of all flags $\{\mathbb{V}_i\}_{i=1}^d$  with $\dim (\mathbb{V}_i) = n_i$, $i=1, \dots, d$, is a smooth manifold called a \emph{flag manifold}, and denoted by $\Flag(n_1,\dots,n_d;\mathbb{V})$. This is  a generalization of the Grassmannian $\Gr(k, \mathbb{V})$ that parameterizes $k$-dimensional linear subspaces in $\mathbb{V}$  as flags of length one are just subspaces, i.e.,  $\Flag(k; \mathbb{V}) =\Gr(k, \mathbb{V})$. Flag manifolds,  sometimes also called flag varieties, were first studied by Ehresmann \cite{Ehresmann1934} and saw rapid development in 1950's \cite{borel1953b,borel1953a,Borel1953,chern1953}. They are now ubiquitous in many areas of pure mathematics, and, as we will discuss next, they are also ubiquitous in applied mathematics, just hidden in plain sight.

The optimization algorithms on Grassmann and Stiefel manifolds originally proposed in \cite{EAS} have found widespread applications: e.g., computer vision \cite{WWF,WSCG}, shape analysis \cite{WW2012,Schulz2014}, matrix computations \cite{B2013,LL2001}, subspace tracking \cite{BNR}, and numerous other areas --- unsurprising as subspaces and their orthonormal bases are ubiquitous in all areas of science and engineering. For the same reason, we expect optimization algorithms on flag manifolds to be similarly useful as flags are also ubiquitous --- any multilevel, multiresolution, or multiscale phenomena likely involve flags, whether implicitly or explicitly. We will discuss some examples from numerical analysis and statistics.

\subsection{Flags in numerical analysis}
In numerical analysis, flags naturally arise in finite elements, multigrid, spectral and pseudospectral methods, wavelets, iterative matrix computations, etc, in several ways.
\begin{example}[Refining mesh]
In multigrid, algebraic multigrid, finite element methods, we often consider a sequence of increasingly finer \emph{grids} or \emph{meshes} $G_1 \subseteq G_2
\subseteq G_3 \subseteq \cdots$ on the domain of interest $\Omega$. The vector space of real-valued functions
\[
\mathbb{V}_k \coloneqq \{ f : G_k \to \mathbb{R} \}
\]
gives us a flag $\mathbb{V}_1 \subseteq \mathbb{V}_2 \subseteq  \mathbb{V}_3 \subseteq \cdots$ of finite-dimensional vector spaces where $\dim \mathbb{V}_k = \# G_k$. The aforementioned numerical methods are essentially different ways of extracting approximate solutions of increasing accuracy  from the flag.
\end{example}

\begin{example}[Increasing order]
In  spectral and pseudospectral methods, we consider a class of functions of increasing complexity determined by an \emph{order} $d$, e.g., polynomial or trigonometric polynomial functions of degree $d$, on the domain of interest $\Omega$. The vector space
\[
\mathbb{V}_d \coloneqq \{ f : \Omega \to \mathbb{R} : \degree(f) \le d \}
\]
gives us a flag  $\mathbb{V}_1 \subseteq \mathbb{V}_2 \subseteq  \mathbb{V}_3 \subseteq \cdots$ as $d$ is increased. Again, these methods operate by extracting approximate solutions of increasing accuracy  from the flag.
\end{example}

\begin{example}[Cyclic subspaces]
Given $A \in \mathbb{R}^{n \times n}$ and $b \in \mathbb{R}^n$, the subspace
\[
K_{k}(A,b) \coloneqq \operatorname{span}\{b,Ab,\dots,A^{k-1} b\}
\]
is called 
the $k$th \emph{Krylov subspace}. The gist behind Krylov subspace methods in numerical linear algebra, whether for computing solutions to linear systems, least squares problems, eigenvalue problems, matrix functions, etc, are all based on finding a sequence of increasingly better approximations from the flag $K_0(A,b) \subseteq K_1(A,b) \subseteq \dots \subseteq K_k(A,b)$.
\end{example}

\begin{example}[Multiresolution]
A standard way to construct wavelets is to define a \emph{multiresolution analysis}, i.e., a sequence of subspaces $\mathbb{V}_{k+1} \subseteq \mathbb{V}_k$ defined by
\[
f(t) \in \mathbb{V}_k \quad \Leftrightarrow \quad f(t/2) \in \mathbb{V}_{k+1}.
\]
The convention in wavelet literature has the indexing in reverse order but this is a minor matter --- a nested of sequence of subspaces is a flag regardless of how the subspaces in the sequence  are labeled. So a multiresolution analysis is also a flag.
\end{example}

This is not an exhaustive list, flags also arise in numerical analysis in other ways, e.g., analysis of eigenvalue methods \cite{GM1986, JH2002}.

\subsection{Flags in statistics} Although not usually viewed in this manner, classical multivariate data analysis techniques \cite{MKB} may be cast as nested subspace-searching problems, i.e., constrained or unconstrained optimization problems on the flag manifold.

We let $\mathbbm{1}$ denote a vector of all ones (of appropriate dimension). We assume that our data set is given in the form of a sample-by-variables design matrix $X \in \mathbb{R}^{n \times p}$, $n \ge p$, which we call a data matrix for short. Let $\overline{x}=\frac{1}{n}X^\tp \mathbbm{1}\in\mathbb{R}^{p}$  be its sample mean and $S_X = (X-\mathbbm{1}\overline{x}^\tp)^\tp (X-\mathbbm{1}\overline{x}^\tp) \in \mathbb{R}^{p \times p}$ be its sample covariance. For another data matrix $Y \in \mathbb{R}^{n \times q}$, $S_{XY} = (X-\mathbbm{1}\overline{x}^\tp)^\tp (Y-\mathbbm{1}\overline{y}^\tp) = S_{YX}^\tp \in \mathbb{R}^{p \times q}$ denotes sample cross-covariance.

\begin{example}[Principal Component Analysis (PCA)]\label{eg:pca}
The $k$th \emph{principal subspace} of $X$ is $\im(Z_k)$, where $Z_k$ is the $p \times k$ orthonormal matrix given by
\begin{equation}\label{eq:pca}
Z_k = \argmax \{ \tr(Z^\tp S_X Z)  : Z \in \V(k,p)\}, \quad k =1,\dots, p.
\end{equation}
So  $\im(Z_k)$ is a $k$-dimensional linear subspace of $\mathbb{R}^p$ spanned by the orthonormal columns of $Z_k$. In an appropriate sense, the $k$th principal subspace captures the greatest variability in the data among all  $k$-dimensional subspaces of $\mathbb{R}^p$. In principal component analysis (PCA), the data points, i.e., columns of $X$, are often projected onto $\im(Z_k)$ with $k =2,3$ for visualization or with other small values of $k$ for dimension reduction. Clearly $\im(Z_k)$ is contained in $\im(Z_{k+1})$ and the flag
\[
\im(Z_1) \subseteq  \im(Z_2) \subseteq \dots \subseteq \im(Z_p)
\]
explains an increasing amount of variance in the data.
\end{example}
In \cite[Theorem~9]{PX}, it is shown how one may directly define PCA as an optimization problem on a flag manifold, a powerful perspective that in turn allows one to generalize and extend PCA in various manners. Nevertheless what is lacking in \cite{PX} is an algorithm for optimization on flag manifolds, a gap that our article will fill.

\begin{example}[Canonical Correlation Analysis (CCA)]\label{eg:cca}
The $k$th pair of \emph{canonical correlation loadings} $(a_k, b_k) \in \mathbb{R}^p \times \mathbb{R}^q$ is defined recursively by
\begin{multline}\label{eq:cca}
(a_k, b_k) = \argmax \{ a^\tp  S_{XY} b : a^\tp S_{X} a = b^\tp S_{Y} b = 1, \\
 a^\tp S_X a_j = a^\tp S_{XY} b_j = b^\tp S_{YX} a_j = b^\tp S_Y b_j = 0, \; j =1,\dots,k-1 \}.
\end{multline}
Let $A_k = [a_1,\dots,a_k] \in \mathbb{R}^{p \times k}$ and $B_k = [b_1,\dots,b_k] \in \mathbb{R}^{q \times k}$. Then the canonical correlation subspaces of $X$ and $Y$ are given by
\[
\im(A_1) \subseteq \dots \subseteq \im(A_p) \quad\text{and}\quad \im(B_1) \subseteq \dots \subseteq \im(B_q),
\]
which are flags in $\mathbb{R}^p$ and   $\mathbb{R}^q$ respectively. Collectively they capture how the  shared variance between the two data sets increases with $k$.
\end{example}

\begin{example}[Correspondence Analysis (CA)]\label{eg:ca}
Let $t = \mathbbm{1}^\tp X\mathbbm{1} \in \mathbb{R}$, $r = \frac{1}{t} X\mathbbm{1}  \in \mathbb{R}^n$, $c = \frac{1}{t} X^\tp \mathbbm{1} \in \mathbb{R}^p$ denote the total, row, and column weights of $X$ respectively and set $D_r = \frac{1}{t} \diag(r) \in \mathbb{R}^{n \times n}$, $D_c = \frac{1}{t} \diag (c) \in \mathbb{R}^{p \times p}$. For $k =1,\dots,p$, we seek matrices $U_k \in \mathbb{R}^{k \times n}$ and $V_k \in \mathbb{R}^{k \times p}$ such that
\begin{equation}\label{eq:ca}
(U_k, V_k) = \argmax   \{ \tr \bigl( U^\tp (\tfrac{1}{t} X - rc^\tp)  V \bigr) :  U^\tp D_r U = I = V^\tp D_c V \}.
\end{equation}
The solution
\[
\im(U_1) \subseteq \dots \subseteq \im(U_p) \quad\text{and}\quad \im(V_1) \subseteq \dots \subseteq \im(V_p)
\]
are flags in $\mathbb{R}^n$ and   $\mathbb{R}^p$ respectively and collectively they explain the increasing deviation from the independence of occurrence of two outcomes.
\end{example}

For  reasons such as  sensitivity of the higher-dimensional subspaces to noise in the data, in practice one relies on the first few subspaces in these flags to make various inference about the data. Nevertheless, we stress that the respective flags that solve \eqref{eq:pca}, \eqref{eq:cca}, \eqref{eq:ca} over all $k$ will paint a complete picture showing the full profile of how variance, shared variance, or deviation from independence vary across dimensions.

Apart from PCA, CCA, and CA, flags arise in other multivariate data analytic techniques \cite{MKB}, e.g., factor analysis (FA), linear discriminant analysis (LDA), multidimensional scaling (MDS), etc, in much the same manner. One notable example is the  independent subspace analysis proposed in \cite{NSP2006,NAP2008},  a generalization of independent component analysis.

\subsection{Prior work and our contributions} Some elements of optimization theory on flag manifolds have been considered in \cite{NSP2006}, although optimization is not its main focus and only analytic expressions for tangent spaces and gradients have been obtained. In particular, no actual algorithm appears in \cite{NSP2006} --- note that a  Riemannian steepest descent algorithm in the spirit of \cite{EAS} would at least require analytic expressions for  geodesics and, to the best of our knowledge,  they have never been derived; in fact prior to this article it is not even known if such expressions exist.

The main contribution of  our article is in providing all necessary ingredients for optimization algorithms on flag manifolds in full details, and from two different perspectives --- representing a flag manifold as (i) a homogeneous space, where a flag is represented as an equivalence class of matrices; and as (ii) a compact submanifold of $\mathbb{R}^{n \times n}$, where every flag is uniquely represented by a matrix. We will provide four systems of extrinsic coordinates for representing a flag manifold that arise from (i) and (ii) --- while modern differential geometry invariably adopts an intrinsic coordinate-free approach, we emphasize that such suitable extrinsic coordinate systems are indispensable for performing computations on manifolds.

In particular, the analytic expressions for various differential geometric objects and operations required for our optimization algorithms will rely on these coordinate systems. We will supply ready-to-use formulas and algorithms, rigorously proven but also made accessible to applied mathematicians and practitioners. For the readers' convenience, the following is a road map to the formulas and algorithms:
\begin{center}
\footnotesize
\begin{tabular}{l|l}
\textsc{object on flag manifold} & \textsc{results}\\\hline\hline
point & Propositions~\ref{prop:equivalent definition of flag manifold}, \ref{prop:Stiefel}, \ref{prop:representation of flags}, \ref{prop:reduced coord}\\\hline
tangent vector & Propositions~\ref{prop:tangent1}, \ref{prop:tangent2a}, \ref{prop:tangent2}, \ref{prop:tangent3}, Corollary~\ref{prop:tangent2b}\\\hline
metric & Propositions~\ref{prop:metric}, \ref{prop:metric2}, \ref{prop:metric3} \\\hline
geodesic & Propositions \ref{prop:geodesics1},  \ref{prop:geodesics2}, \ref{prop:geodesic3a}, \ref{prop:geodesics3}\\\hline
arclength & Corollary~\ref{cor:arc-length}, Proposition~\ref{prop:geodesic3a}\\\hline
geodesic distance & Proposition~\ref{prop:geodist}\\\hline
parallel transport & Propositions \ref{prop:pt}, \ref{prop:pt2}, \ref{prop:pt1}\\\hline
gradient & Proposition~\ref{prop:gradient}\\\hline
Hessian & Proposition~\ref{prop:Hessian}\\\hline
steepest descent & Algorithm~\ref{alg:sd}\\\hline
conjugate gradient  & Algorithm~\ref{alg:cg}
\end{tabular}
\end{center}

\subsection{Outline} We begin by reviewing basic materials about Lie groups, Lie algebras, homogeneous spaces, and Riemannian manifolds (Section~\ref{sec:basic}). We then proceed to describe the basic differential geometry of flag manifolds (Section~\ref{sec:flag}), develop four concrete matrix representations of flag manifolds, and derive closed-form analytic expressions for various differential geometric objects in terms of standard matrix operations (Sections~\ref{sec:homo}, \ref{sec:mfld}, \ref{sec:gH}). With these, standard nonlinear optimization algorithms can be ported to the flag manifold almost as an afterthought (Section~\ref{sec:optim}). We illustrate using two numerical experiments with steepest descent on the flag manifold (Section~\ref{sec:num}).

\section{Basic differential geometry of homogeneous spaces}\label{sec:basic}

We will need some rudimentary properties of homogeneous spaces not typically found in the manifold optimization literature, e.g., \cite{AMS2009,EAS}.  This section provides a self-contained review, kept to a bear minimum of just what we need later. We refer readers to standard references \cite{Helgason1978,KN1969,Boothby1975} for more information.

\subsection{Lie groups and Lie algebras}\label{sec:Lie}
Let $M$ be a smooth manifold and $T^\ast M$ be its cotangent bundle. A \emph{Riemannian metric} on $M$ is a smooth section $g: M \to T^\ast M \otimes T^\ast M$ such that $g_x \coloneqq g(x) \in T^\ast_xM \otimes T^\ast_x M$ is a positive definite symmetric bilinear form on the tangent space $T_xM$ for every $x\in M$. Intuitively, a Riemannian metric gives an inner product on $T_xM$ for every $x\in M$ and it varies smoothly with respect to $x\in M$. Let $G$ be a group and let $m:G\times G \to G$ be the multiplication map $m(a_1,a_2) = a_1a_2$ and $i: G \to G$ be the inversion map $i(a) = a^{-1}$. Then $G$ is a \emph{Lie group} if it is a smooth manifold and the group operations $m$ and $i$ are smooth maps. The tangent space $\mathfrak{g}$ of $G$ at the identity $e\in G$ is a \emph{Lie algebra}, i.e., a vector space equipped with a \emph{Lie bracket}, a bilinear map  $[\cdot, \cdot]: \mathfrak{g} \times \mathfrak{g} \to \mathfrak{g}$ satisfying $[X,Y] = -[Y,X]$ (skew-symmetry) and $[X,[Y,Z]] + [Z,[X,Y]] + [Y,[Z,X]] = 0$ (Jacobi identity). For example, if $G$ is the orthogonal group $\O(n)$ of all $n\times n$ real orthogonal matrices, then its Lie algebra  $\mathfrak{so}(n)$ is the vector space of all $n\times n$ real skew-symmetric matrices.

For a Lie group $G$, we may define the \emph{left} and \emph{right translation} maps $L_a,R_a:G \to G$ by $L_a(x) = m(a,x) = ax$ and $R_a(x) = m(x,a) = xa$. We say that a Riemannian metric $g$ on $G$ is  \emph{left invariant} if for all $a \in G$,
\[
g_{L_a(x)} \bigl((dL_a)_x (X), (dL_a)_x (Y)\bigr) = g_x(X,Y);
\]
\emph{right invariant} if for all $b \in G$,
\[
g_{R_b(x)} \bigl((dR_b)_x (X), (dR_b)_x (Y)\bigr) = g_x(X,Y);
\]
and \emph{bi-invariant} if for all $a,b\in G$,
\[
g_{R_b \circ L_a(x)} \bigl( (d (R_b \circ L_a))_x (X), (d (R_b \circ L_a))_x (Y)\bigr) = g_x(X,Y)
\]
over all $X,Y\in T_xM$.

\subsection{Homogeneous spaces}
We now recall some basic definitions and facts about homogeneous spaces. Throughout this article, we will use double brackets $\lb x\rb$ to denote the equivalence class of $x$.
\begin{definition}
Let $G$ be a Lie group acting on a smooth manifold $M$ via $\varphi: G\times M \to M$. If the action $\varphi$ is smooth and transitive, i.e.,  for any $x,y\in M$, there is some $a\in G$ such that $\varphi(a,x) = y$, then $M$ is called a \emph{homogeneous space} of the Lie group $G$.
\end{definition}
For a point $x\in M$, the subgroup $G_x =\{a\in G: \varphi(a,x) = x\}$ is called the \emph{isotropy group} of $x$. We write $G/G_x$ for the quotient group of $G$ by $G_x$ and denote by $\lb a \rb\in G/G_x$ the \emph{coset} (or equivalence class) of $a\in G$. Since $G$ acts on $M$ transitively, we see that there is  a one-to-one correspondence $F$ between $G/G_{x}$ and $M$ given by
\[
F: G/G_x \to M,\quad F(\lb a\rb) = \varphi(a,x),
\] 
for any $x\in M$. In fact, $F$ defines a diffeomorphism between the two smooth manifolds, which is the content of the following theorem \cite[Theorems~9.2 and 9.3]{Boothby1975}.

\begin{theorem}\label{thm:homogeneous structure}
Let $G$ be a Lie group acting on a smooth manifold $M$. For any $x \in M$, there exists a unique smooth structure on $G/G_x$ such that the action 
\[
\psi: G \times G/G_x \to G, \quad \psi(a, \lb a' \rb) = \lb aa' \rb 
\]
is smooth. Moreover, the map $F: G/G_x \to M$ sending $\lb a \rb$ to $\varphi(a,x)$ is a $G$-equivariant diffeomorphism, i.e., $F$ is a diffeomorphism such that $F\bigl(\psi(a, \lb a' \rb)\bigr) = \varphi \bigl(a, F(\lb a' \rb)\bigr)$.
\end{theorem}
The Grassmannian $\Gr(k,n)$ of $k$-dimensional subspaces in $\mathbb{R}^n$ is probably the best known example of a homogeneous space in manifold optimization. Indeed, $\O(n)$ acts transitively on $\Gr(k,n)$ and as any $k$-dimensional subspace $\mathbb{W} \subseteq \mathbb{R}^n$ has isotropy group  isomorphic to $\O(k)\times \O(n-k)$, we obtain the well-known characterization of Grassmannian
\[
\Gr(k,n) \cong \O(n)/\bigl(\O(k)\times \O(n-k)\bigr)
\]
that is crucial for manifold optimization. Throughout this article `$\cong$' will mean diffeomorphism.

Let $G$ be a Lie group and $M$ a homogeneous space of $G$ with action $\varphi: G\times M \to M$. Fix any $x\in M$ and let $H$ denote its isotropy group. By Theorem~\ref{thm:homogeneous structure} we may identify $M$ with $G/H$.  The left translation map in Section~\ref{sec:Lie} may be extended to this setting as $L_a : M \to M$, $L_a(y) =\varphi(a,y)$ for any  $a \in G$. In particular, if $a\in H$, then $L_a(x) = x$, and we have a linear isomorphism
\[
(dL_a)_x : T_xM \to T_xM.
\]
Let $g: M \to T^\ast M \otimes T^\ast M$ be a Riemannian metric on $M$. We say that $g$ is $G$-\emph{invariant} if for every $y\in M$ and $a\in G$, we have
\[
g_{L_a(y)}\bigl((dL_a)_y(X), (dL_a)_y(Y)\bigr) = g_y (X,Y) \quad \text{for all }  X,Y\in T_yM.
\]

As $M  = G/H$, we have $T_{\lb e \rb} M  = \mathfrak{g}/\mathfrak{h}$ where $\mathfrak{g}$ and $\mathfrak{h}$ are the Lie algebras of $G$ and $H$ respectively. Here $e\in G$ is the identity element. This allows us to define the \emph{adjoint representation}
$\Ad_H: H \to \GL(\mathfrak{g}/\mathfrak{h})$, $a \mapsto d(L_a \circ R_{a^{-1}})_{\lb e \rb}$. In other words, for any $a \in H$ and $X\in \mathfrak{g}/\mathfrak{h}$,
\[
\Ad_H (a) (X) = d (L_a\circ R_{a^{-1}} )_{\lb e \rb}(X).
\]
An inner product $\eta$ on the vector space $\mathfrak{g}/\mathfrak{h}$ is said to be \emph{$\Ad_H$-invariant} if for every $a\in H$,
\[
\eta (\Ad_H(a)(X),\Ad_H(a)(Y)) = \eta (X,Y) \quad \text{for all } X,Y\in \mathfrak{g}/\mathfrak{h}.
\]

We state an important result about their existence and construction \cite[Proposition 3.16]{CE2008}.

\begin{proposition}\label{prop:metric inner product correspondence}
Let $G$ be a connected Lie group and  $H$ a closed Lie subgroup with Lie algebras $\mathfrak{g}$ and $\mathfrak{h}$ respectively. If there is a subspace $\mathfrak{m}$ of $\mathfrak{g}$ such that $\mathfrak{g} = \mathfrak{m} \oplus \mathfrak{h}$ and $\Ad_H (\mathfrak{m}) \subseteq \mathfrak{m}$, then there is a one-to-one correspondence between $G$-invariant metrics on $M = G/H$ and $\Ad_H$-invariant inner products on  $\mathfrak{m}$.
\end{proposition} 
Proposition~\ref{prop:metric inner product correspondence} says that if $\mathfrak{h}\subseteq \mathfrak{g}$ admits a complement $\mathfrak{m}$, then we may obtain a $G$-invariant metric $g$ on $M$ by an $\Ad_H$-invariant  inner product on $\mathfrak{m}$. Moreover, we may identify $T_xM$ with $\mathfrak{m}$, implying that the metric $g$ on $M$ is essentially determined by $g_x$ at a single arbitrary point $x \in M$. 

If in addition $G$ is simple and compact, then $G$ admits the unique bi-invariant metric  called the \emph{canonical metric} on $M$ and $(M,g)$ is called a \emph{normal homogeneous space}.
\begin{proposition}\label{prop:nhs}
If $G$ is a compact Lie group, then $G$ admits a bi-invariant metric and this metric induces a $G$-invariant metric $g$ on $M = G/H$ for any closed subgroup $H\subseteq G$.
\end{proposition}

\subsection{Geodesic orbit spaces}
Let $M = G/H$ be a homogeneous space of $G$. If $M$ has a Riemannian metric $g$ such that every geodesic in $M$ is an orbit of a one-parameter subgroup of $G$, then we say that $(M,g)$ is a \emph{geodesic orbit space}. The following result \cite{KS2000} will allow us to construct several interesting examples.
\begin{theorem}\label{thm:GO}
Let $G$ be a compact Lie group with a bi-invariant metric $g$ and  $H$ be a subgroup such that $M = G/H$ is a smooth manifold (e.g., $H$ is closed subgroup). Then $M = G/H$ together with the metric $\widetilde{g}$ induced by $g$ is a geodesic orbit space.
\end{theorem}

In general it is difficult if not impossible to determine closed form analytic expressions for geodesics on a Riemannian manifold. But in the case of a geodesic orbit space, since its geodesics are simply orbits of one-parameter subgroups of $G$, the task reduces to determining the latter. The next result \cite[Theorem~1.3.5]{GW2009} will be helpful towards this end.
\begin{theorem}\label{thm:1-parameter subgp}
If $G$ is a matrix Lie group equipped with a bi-invariant metric, then every one-parameter subgroup $\gamma(t)$ of $G$ is of the form 
\[
\gamma(t) = \exp (ta) \coloneqq \sum_{k=0}^\infty \frac{t^ka^k}{k!}
\]
for some $a \in \mathfrak{g}$.
\end{theorem}
So for example, every one-parameter subgroup of $\SO(n)$ must take the form $\gamma(t)  =\exp (ta)$ for some skew-symmetric matrix $a\in \mathfrak{so}(n)$.

\subsection{Riemannian notions}\label{sec:Riem}

Although not specific to homogeneous or geodesic orbit spaces, we state the famous Hopf--Rinow theorem \cite[Theorem $1.8$]{CE2008} and recall the definitions of Riemannian gradient and Hessian \cite{Helgason1978,KN1969,Boothby1975} below for easy reference.
\begin{theorem}[Hopf--Rinow]\label{thm:Hopf--Rinow}
Let $(M,g)$ be a connected Riemannian manifold. Then the following statements are equivalent:
\begin{enumerate}[\upshape (i)]
\item closed and bounded subsets of $M$ are compact;
\item $M$ is a complete metric space;
\item $M$ is geodesically complete, i.e., the exponential map $\exp_x: T_xM \to M$ is defined on the whole $T_xM$ for all $x\in M$.
\end{enumerate}
Furthermore, any one of these conditions guarantees that any two points $x,y$ on $M$ can be connected by a distance minimizing geodesic on $M$.
\end{theorem}
In the following we will write $ \mathfrak{X}(M)$ for the set of all smooth vector fields on $M$.
\begin{definition}[Riemannian gradient and Hessian]\label{def:gH}
Let $(M,g)$ be a Riemannian manifold. Let $f:M \to \mathbb{R}$ be a smooth function. The \emph{Riemannian gradient} of $f$, denoted $\nabla f$, is defined by
\[
g(\nabla f,V) = V(f),
\]
for any $V \in \mathfrak{X}(M)$. The \emph{Riemannian Hessian}  of $f$, denoted $\nabla^2 f$, is defined by
\[
(\nabla^2 f)(U,V) = g\bigl(\nabla_U(\nabla f), V\bigr),
\]
where $U,V \in \mathfrak{X}(M)$ and $\nabla_U V$ is the \emph{covariant derivative} of $V$ along $U$, which is uniquely determined by the Riemannian metric $g$.
\end{definition}

By their definitions, $\nabla f$ is a smooth vector field and $\nabla^2 f$ is a smooth field of \emph{symmetric bilinear forms}. In particular, $\nabla^2 f$ is uniquely determined by its values at points of the form $(V,V)$ over all $V\in \mathfrak{X}(M)$ because of bilinearity and symmetry, i.e.,
\begin{equation}\label{eq:polar}
\nabla^2f(U,V) = \frac{1}{2} \bigl(\nabla^2f (U + V, U + V)  -\nabla^2f(U,U) - \nabla^2f(V,V) \bigr),
\end{equation}
for any $U,V  \in \mathfrak{X}(M)$. Definition~\ref{def:gH} is standard but not as useful for us as a \emph{pointwise} definition --- the Riemannian gradient $\nabla f(x)$ and Riemannian Hessian $\nabla^2f(x)$ at  a point  $x\in M$ is given by
\begin{equation}\label{eq:gH}
g_x(\nabla f(x),X) = \frac{d f\bigl(\exp(tX)\bigr)}{dt}\Bigr|_{t=0},\qquad 
\nabla^2 f(x) (X,X) =\frac{d^2 f\bigl(\exp(tX)\bigr)}{dt^2}\Bigr|_{t=0},
\end{equation}
where $\exp(tX)$ is the geodesic curve emanating from $x$ in the direction $X \in T_x M$. We may obtain \eqref{eq:gH} by Taylor expanding $f\bigl(\exp(tX)\bigr)$.

Given a specific function $f$, one may express \eqref{eq:gH} in terms of \emph{local coordinates} on $M$ but in general there are no global formulas for  $\nabla f(x)$ and $\nabla^2f(x)$, and without which it would be difficult if not impossible to do optimization on $M$. We will see in Section~\ref{sec:gH} that when $M$ is a flag manifold, then  \eqref{eq:gH} may be expressed globally in terms of extrinsic coordinates.

\section{Basic differential geometry of flag manifolds}\label{sec:flag}

We will now define flags and flag manifolds formally and discuss some basic properties. Let $n$ be a positive integer and $\mathbb{V}$ be an $n$-dimensional vector space over $\mathbb{R}$. We write $\V(k, \mathbb{V})$ for the Stiefel manifold \cite{Stiefel1935} of $k$-frames in $\mathbb{V}$ and $\Gr(k, \mathbb{V})$ for the Grassmannian \cite{Grass} of $k$-dimensional subspaces in $\mathbb{V}$. If the choice of $\mathbb{V}$ is unimportant or if $\mathbb{V} = \mathbb{R}^n$, then we will just write $\V(k,n)$ and $\Gr(k,n)$.
\begin{definition}
Let $0 < n_1 < \dots < n_d < n$ be an increasing sequence of $d$ positive integers and $\mathbb{V}$ be an $n$-dimensional vector space over $\mathbb{R}$. A flag of type $(n_1,\dots, n_d)$ in $\mathbb{V}$ is a sequence of subspaces 
\[
\mathbb{V}_1 \subsetneq  \mathbb{V}_2 \subsetneq \dots \subsetneq \mathbb{V}_d, \quad \dim \mathbb{V}_i = n_i,\quad i=1,\dots,d.  
\]
We denote the set of such flags by $\Flag(n_1,\dots,n_d;\mathbb{V})$ and call it  the \emph{flag manifold} of type $(n_1,\dots, n_d)$. If $\mathbb{V}$ is unimportant or if $\mathbb{V} = \mathbb{R}^n$, then we will just write $\Flag(n_1,\dots, n_d;n)$.
\end{definition}
For notational convenience we will adopt the following convention throughout:
\[
n_0 \coloneqq 0, \qquad n_{d+1} \coloneqq n, \qquad \mathbb{V}_0 \coloneqq \{0 \}, \qquad \mathbb{V}_{d+1} \coloneqq \mathbb{V}.
\]

We will see in Proposition~\ref{prop:geometric structure} that flag manifolds are indeed manifolds. When $d=1$, $\Flag(k; \mathbb{V})$ is  the set of all $k$-dimensional subspaces of $\mathbb{V}$, which is the Grassmannian $\Gr(k,\mathbb{V})$. The other extreme case is when $d= n-1$ and $n_i = i$, $i =1,\dots, n-1$, and in which case $\Flag(1,\dots, n-1;\mathbb{V})$ comprises all \emph{complete flags} of $\mathbb{V}$, i.e.,
\[
\mathbb{V}_1 \subsetneq  \mathbb{V}_2 \subsetneq \dots \subsetneq \mathbb{V}_{n-1}, \quad \dim \mathbb{V}_i = i,\quad i=1,\dots,n-1.  
\]

Like the Grassmannian, the flag manifold is not merely a set but has rich geometric structures. We will start with the most basic ones and defer other useful characterizations to Sections~\ref{sec:homo} and \ref{sec:mfld}.
\begin{proposition}\label{prop:geometric structure}
Let $0 < n_1< \dots < n_d < n$ be integers and $\mathbb{V}$ be an $n$-dimensional real vector space. The flag manifold $\Flag(n_1,\dots, n_d;\mathbb{V})$ is
\begin{enumerate}[\upshape (i)]
\item\label{it:mfld} a connected compact smooth manifold;
\item \label{it:var1} an irreducible affine variety;
\item\label{it:cls1} a closed submanifold of $\Gr(n_1,\mathbb{V})\times \Gr(n_2,\mathbb{V}) \times \dots \times \Gr(n_d,\mathbb{V})$;
\item\label{it:cls2} a closed submanifold of $\Gr(n_1,\mathbb{V})\times \Gr(n_2-n_1,\mathbb{V}) \times \dots \times \Gr(n_d-n_{d-1},\mathbb{V}) $;
\item\label{it:fiber} a fiber bundle on $\Gr(n_d,\mathbb{V})$ whose fiber over $\mathbb{W}\in \Gr(n_d,\mathbb{V})$ is $\Flag(n_1,\dots, n_{d-1};\mathbb{W})$;
\item\label{it:var2} a smooth projective variety.
\end{enumerate}
\end{proposition}
\begin{proof}
Property~\eqref{it:mfld} is well-known \cite{Monk1959,Borel1991} but also follows from the characterization in Proposition~\ref{prop:equivalent definition of flag manifold} as a quotient of a compact connected Lie group by a closed subgroup. Property~\eqref{it:var1} is a consequence of Propositions~\ref{prop:representation of flags} and \ref{prop:reduced coord}, where we give two different ways of representing $\Flag(n_1,\dots,n_d;\mathbb{V})$ as an affine variety in $\mathbb{R}^m$, $m = (nd)^2$.  Property~\eqref{it:var2} is a consequence of \eqref{it:cls1} or \eqref{it:cls2}, given that the Grassmannian is a projective variety.

In the following, let $\{\mathbb{V}_i\}_{i=1}^d \in \Flag(n_1,\dots,n_d;\mathbb{V})$, i.e., $\dim \mathbb{V}_i = n_i$, $i=1,\dots, d$. For \eqref{it:cls1}, the map 
\begin{equation}\label{eq:j}
\varepsilon : \Flag(n_1,\dots,n_d;\mathbb{V}) \to \Gr(n_1,\mathbb{V})\times \dots \times \Gr(n_d,\mathbb{V}), \quad \{\mathbb{V}_i\}_{i=1}^d \mapsto (\mathbb{V}_1, \mathbb{V}_2, \dots, \mathbb{V}_d)
\end{equation}
is clearly an embedding. Its image is closed since if $(\mathbb{V}_1 , \dots, \mathbb{V}_d) \not\in \varepsilon \bigl(\Flag(n_1,\dots,n_d;\mathbb{V})\bigr)$,
then there exists some $i \in \{ 1,\dots, d-1\}$ such that $\mathbb{V}_i \not\subseteq \mathbb{V}_{i+1}$; so if $\mathbb{V}'_i\in \Gr(n_i,\mathbb{V})$ and $\mathbb{V}'_{i+1} \in \Gr(n_{i+1},\mathbb{V})$ are in some small neighborhood of $\mathbb{V}_i $ and $\mathbb{V}_{i+1}$ respectively, then $\mathbb{V}'_i \not\subseteq \mathbb{V}'_{i+1}$.

For \eqref{it:cls2}, choose and fix an inner product on $\mathbb{V}$. Let $\mathbb{V}_i^{\perp}$ denote the orthogonal complement of $\mathbb{V}_i$ in $\mathbb{V}_{i+1}$, $i=1,\dots, d-1$.
The map
\begin{equation}\label{eq:j2}
\begin{aligned}
\varepsilon ': \Flag(n_1,\dots, n_d;\mathbb{V}) &\to \Gr(n_1,\mathbb{V})\times \Gr(n_2-n_1,\mathbb{V}) \times \dots \times \Gr(n_d-n_{d-1},\mathbb{V}),\\
\{\mathbb{V}_i\}_{i=1}^d &\mapsto (\mathbb{V}_1, \mathbb{V}_1^{\perp}, \dots, \mathbb{V}_{d-1}^{\perp})
\end{aligned}
\end{equation}
is clearly an embedding. That the image of $\varepsilon '$ is closed follows from the same argument used for $\varepsilon$.

For \eqref{it:fiber}, consider the map 
\[
\rho: \Flag(n_1,\dots,n_d;\mathbb{V}) \to \Gr(n_d,\mathbb{V}), \quad \{\mathbb{V}_i\}_{i=1}^d \mapsto \mathbb{V}_d,
\]
that is clearly surjective and smooth. For any  $\mathbb{W} \in  \Gr(n_d,\mathbb{V})$, $\rho^{-1}(\mathbb{W})$ consists of flags of the form 
\[
\mathbb{V}'_1 \subseteq \mathbb{V}'_2 \subseteq \dots \subseteq \mathbb{V}'_{d-1} \subseteq \mathbb{W},\quad \dim \mathbb{V}'_i = n_i,\quad i=1,\dots, d-1.
\]
In other words, the fiber $\rho^{-1}(\mathbb{W}) \cong  \Flag(n_1,\dots, n_{d-1};\mathbb{W})$.
\end{proof}

The fiber bundle structure in Proposition~\ref{prop:geometric structure}\eqref{it:fiber} may be recursively applied to get
\begin{gather*}
\Flag(n_1,\dots, n_{d-1};n_d) \to \Flag(n_1,\dots, n_d;n) \to \Gr(n_d,n),\\
\Flag(n_1,\dots, n_{d-2};n_{d-1}) \to \Flag(n_1,\dots, n_{d-1};n_d) \to \Gr(n_{d-1},n),
\end{gather*}
and so on, ending in the well-known characterization of the Stiefel manifold as a principal bundle over the Grassmannian
\[
\O(n) \to \V(k,n) \to \Gr(k,n).
\]

In the next two sections, we will see how the flag manifold may be equipped with extrinsic matrix coordinates and be represented as either homogeneous spaces of  matrices (Section~\ref{sec:homo}) or manifolds of matrices (Section~\ref{sec:mfld}) that in turn give closed-form analytic expressions for various differential geometric objects and operations needed for  optimization algorithms.

\section{Flag manifolds as matrix homogeneous spaces}\label{sec:homo}

We will discuss three representations of the flag manifold as \emph{matrix homogeneous spaces}, i.e., where a flag is represented as an equivalence class of matrices:
\begin{align}
\Flag(n_1,\dots,n_d;n) &\cong \O(n) / \bigl(\O(n_1) \times \O(n_2 - n_1) \times \dots \times \O(n_d - n_{d-1}) \times \O(n-n_d) \bigr), \label{eq:homo1}\\
\Flag(n_1,\dots,n_d;n) &\cong \SO(n) / \S\bigl(\O(n_1) \times \O(n_2 - n_1) \times \dots \times \O(n_d - n_{d-1}) \times \O(n-n_d) \bigr),\label{eq:homo2}\\
\Flag(n_1,\dots,n_d;n) &\cong  \V(n_d,n) / \bigl(\O(n_1) \times \O(n_2-n_1) \times \dots \times \O(n_d- n_{d-1})\bigr).\label{eq:homo3}
\end{align}
The characterization \eqref{eq:homo1} is standard \cite{Monk1959,Borel1991} and generalizes the well-known characterization of the Grassmannian as $\Gr(k,n) \cong \O(n)/\bigl(\O(k) \times \O(n-k)\bigr)$ whereas the characterization \eqref{eq:homo3} generalizes another well-known characterization of the Grassmannian as $\Gr(k,n) \cong \V(k,n)/\O(k)$.

Nevertheless, we will soon see that it is desirable to describe $\Flag(n_1,\dots,n_d;n)$ as a homogeneous space $G/H$ where $G$ is a connected Lie group --- note that $\O(n)$ is not connected whereas $\V(n_d,n)$ is not a group, so \eqref{eq:homo1} and \eqref{eq:homo3} do not meet this criterion. With this in mind, we state and prove \eqref{eq:homo2} formally.
\begin{proposition}\label{prop:equivalent definition of flag manifold}
Let $0 < n_1< \dots < n_d < n$ be  $d$ positive integers. The flag manifold $\Flag(n_1,\dots, n_d;n)$ is diffeomorphic to the homogeneous space
\[
\SO(n) / \S\bigl(\O(n_1) \times \O(n_2 - n_1) \times \dots \times \O(n_d - n_{d-1}) \times \O(n-n_d) \bigr)
\]
where  $\S\bigl(\O(n_1) \times \O(n_2 - n_1) \times \dots \times \O(n_d - n_{d-1}) \times \O(n-n_d) \bigr)$ is the subgroup of unit-determinant block diagonal matrices with orthogonal blocks, i.e.,
\[
\begin{bsmallmatrix}
Q_1 & 0 & \dots & 0 \\
0 & Q_2 & \dots & 0\\
\vdots & \vdots & \ddots & \vdots \\
0 & 0 & \dots & Q_{d+1}
\end{bsmallmatrix} \in \O(n), \quad Q_i\in \O(n_i-n_{i-1}), \quad i=1,\dots, d+1, \quad \prod_{i=1}^{d+1} \det (Q_i) = 1.
\]
\end{proposition}
\begin{proof}
We start with the characterization \eqref{eq:homo1}, i.e., in this proof we assume `$=$' in place of `$\cong$' in \eqref{eq:homo1}.
We claim that the required diffeomorphism $\tau$ is given as in the commutative diagram below:
\[\begin{tikzcd}
\SO(n) \arrow{r}{j} \arrow[swap]{d}{\pi'} & \O(n) \arrow{d}{\pi} \\
\SO(n) / \S\bigl(\O(n_1) \times \O(n_2 - n_1) \times \dots \times \O(n-n_d) \bigr) \arrow{r}{\tau} & \Flag(n_1,\dots, n_d;n) 
\end{tikzcd}
\]
Here $j$ is the inclusion of $\SO(n)$ in $\O(n)$, $\pi$ and $\pi'$ the respective quotient maps, and $\tau$ the induced map. Since
\[
\SO(n)  \cap  \bigl(\O(n_1) \times \O(n_2 - n_1) \times \dots  \times \O(n-n_d) \bigr)=  \S\bigl(\O(n_1) \times \O(n_2 - n_1) \times \dots \times \O(n - n_{d}) \bigr),
\]
$\tau$ is injective. To show that it is surjective, let $\{\mathbb{V}_i\}_{i=1}^d \in \Flag(n_1,\dots, n_d;n)$ be a flag represented by some $A\in \O(n)$, i.e., $\pi(A) = \{\mathbb{V}_i\}_{i=1}^d$. If $\det(A) = 1$, then we already have $\tau (\pi_1(A)) = \{\mathbb{V}_i\}_{i=1}^d$ by commutativity of the diagram. If $\det(A) = -1$, take any $A_1\in \O(n_1)$ with $\det(A_1)=-1$, set
\[
B = A \begin{bsmallmatrix}
A_1 & 0 &\cdots & 0 \\
0 & I_{n_2 - n_1} & \cdots & 0 \\
\vdots & \vdots & \ddots & \vdots \\
0 & 0 & \cdots & I_{n - n_d}
\end{bsmallmatrix} \in \SO(n),
\]
and observe that $\tau (\pi_1(B)) = \pi (B) = \pi (A)$.
\end{proof}

\subsection{Orthogonal coordinates for the flag manifold}\label{sec:ortho}

An immediate consequence of Proposition~\ref{prop:equivalent definition of flag manifold} is that the flag manifold is connected. The characterization \eqref{eq:homo2} says that a point on $\Flag(n_1,\dots,n_d;n)$  may be represented by  the equivalence class of matrices
\begin{align}\label{eq:equivalence class}
\lb Q \rb = \left\lbrace Q\begin{bsmallmatrix}
Q_1 & 0 & \dots & 0 \\
0 & Q_2 & \dots & 0\\
\vdots & \vdots & \ddots & \vdots \\
0 & 0 & \dots & Q_{d+1}
\end{bsmallmatrix}: Q_i \in \O(n_i-n_{i-1}),\; i=1,\dots,d+1, \;\prod_{i=1}^{d+1} \det Q_i = 1 \right\rbrace
\end{align}
for some $Q \in \SO(n)$. We will call such a representation \emph{orthogonal coordinates} for the flag manifold.

The Lie algebra of $\S\bigl(\O(n_1) \times \O(n_2 - n_1) \times \dots \times \O(n-n_d) \bigr)$ is simply $ \mathfrak{so}(n_1) \times \mathfrak{so}(n_2-n_1)\times \dots \times \mathfrak{so}(n-n_d)$, which we will regard as a Lie subalgebra of block diagonal matrices,
\begin{align}
\mathfrak{h} &= \left\lbrace
\begin{bsmallmatrix}
A_1 & 0 & \dots & 0\\
0 & A_2 & \dots & 0\\
\vdots & \vdots & \ddots & \vdots\\
0 & 0 & \dots & A_{d+1} 
\end{bsmallmatrix} \in \mathfrak{so}(n) :\begin{multlined} A_1 \in \mathfrak{so}(n_1), A_2 \in \mathfrak{so}(n_2-n_1), \dots \\ \dots,  A_{d+1} \in \mathfrak{so}(n-n_d) \end{multlined}  \right\rbrace. \label{eq:h}
\intertext{Let $\mathfrak{m}$ be the natural complement of $\mathfrak{h}$ in $\mathfrak{so}(n)$,}
\mathfrak{m} &= \left\lbrace
\begin{bsmallmatrix}
0 & B_{1,2}  &\dots & B_{1,d+1} \\
-B_{1,2}^\tp  & 0   &\dots & B_{2,d+1} \\
\vdots & \vdots & \ddots & \vdots\\
-B_{1,d+1}^\tp  & -B_{2,d+1}^\tp  & \dots & 0
\end{bsmallmatrix} \in \mathfrak{so}(n) : \begin{multlined} B_{ij}\in \mathbb{R}^{(n_i-n_{i-1}) \times (n_j-n_{j-1})}, \\ 1\le i < j \le d+1 \end{multlined}\right\rbrace.\label{eq:m}
\end{align}
In particular, we have the direct sum decomposition $\mathfrak{so}(n) = \mathfrak{h}  \oplus \mathfrak{m}$ as vector spaces.

The groups $\O(n)$ and $\O(n_1) \times \O(n_2-n_1)\times \cdots \times \O(n - n_d)$ have the same Lie algebras as $\SO(n)$ and $\S\bigl(\O(n_1) \times \O(n_2-n_1)\times  \cdots \times  \O(n - n_d)\bigr)$, namely, $\mathfrak{so}(n)$ and $\mathfrak{so}(n_1) \times \mathfrak{so}(n_2-n_1)\times \dots \times \mathfrak{so}(n-n_d)$ respectively. The tangent space of a homogeneous space $G/H$ at any point is a translation of the tangent space at the identity element $\lb H \rb\in G/H$, which depends only on the Lie algebras $\mathfrak{g}$ and $\mathfrak{h}$ of $G$ and $H$ respectively:
\[
T_{\lb H \rb} G/H  \simeq \mathfrak{g}/\mathfrak{h},
\]
a fact that we will use in the proof of Proposition~\ref{prop:tangent1}. As such we do not need to distinguish the two homogeneous space structures \eqref{eq:homo1} and \eqref{eq:homo2} when we discuss geometric quantities associated with tangent spaces, e.g., geodesic, gradient, Hessian, parallel transport. In the sequel we will make free use of this flexibility in switching between \eqref{eq:homo1} and \eqref{eq:homo2}.

\begin{proposition}\label{prop:reductive complement}
Let $\mathfrak{h}$ and $\mathfrak{m}$ be as in \eqref{eq:h} and \eqref{eq:m} and $H =\O(n_1)\times \O(n_2-n_1)\times \dots \times \O(n-n_d)$. Then the subspace $\mathfrak{m}$ is $\Ad_H$-invariant, i.e., $\Ad(a)(X) \in \mathfrak{m}$ for every $a \in H$ and $X\in \mathfrak{m}$.
\end{proposition}
\begin{proof}
We need to show that $\Ad(a)(X) \in \mathfrak{m}$ whenever $a\in H$ and $X\in \mathfrak{m}$. For notational simplicity, we assume $d = 2$. Let
\[
a = \begin{bsmallmatrix}
A_1 & 0 & 0 \\[0.75ex]
0 & A_2 & 0 \\[0.75ex]
0 & 0 & A_3
\end{bsmallmatrix}\qquad\text{and}\qquad
X = \begin{bsmallmatrix}
0 & B_{1,2} & B_{1,3} \\
-B_{1,2}^\tp  & 0 & B_{2,3} \\
-B_{1,3}^\tp  & -B_{2,3}^\tp  & 0
\end{bsmallmatrix},
\]
where $A_i \in \O(n_i-n_{i-1})$, $i =1,2,3$, and
$B_{ij} \in \mathbb{R}^{(n_i-n_{i-1}) \times (n_j-n_{j-1})}$, $1\le i < j \le 3$.
Then $\Ad(a)(X) =a X a^{-1} = a X a^\tp $ since $a$ is an orthogonal matrix; and we have 
\[
a X a^\tp  = \begin{bsmallmatrix}
A_1 & 0 & 0 \\[0.75ex]
0 & A_2 & 0 \\[0.75ex]
0 & 0 & A_3
\end{bsmallmatrix}  \begin{bsmallmatrix}
0 & B_{1,2} & B_{1,3} \\
-B_{1,2}^\tp  & 0 & B_{2,3} \\
-B_{1,3}^\tp  & -B_{2,3}^\tp  & 0
\end{bsmallmatrix} \begin{bsmallmatrix}
A_1^\tp  & 0 & 0 \\
0 & A_2^\tp  & 0 \\
0 & 0 & A_3^\tp 
\end{bsmallmatrix} 
 = \begin{bsmallmatrix}
0 & A_1 B_{1,2} A_2^\tp  & A_1B_{1,3}A_3^\tp  \\
-A_2 B_{1,2}^\tp  A_1^\tp  & 0 & A_2B_{2,3}A_3^\tp  \\
-A_3B_{1,3}^\tp A_1^\tp  & -A_3B_{2,3}^\tp  A_2^\tp  & 0
\end{bsmallmatrix} \in \mathfrak{m}
\]
as required.
\end{proof}

We  now have all the ingredients necessary for deriving closed-form analytic  expressions for the tangent space, metric, geodesic, geodesic distance, and parallel transport on a flag manifold in orthogonal coordinates. We begin with the representation of a tangent  space as a vector space of matrices.
\begin{proposition}[Tangent space I]\label{prop:tangent1}
Let $\lb Q \rb\in \Flag(n_1,\dots, n_d;n) = \O(n) / \bigl(\O(n_1) \times \O(n_2 - n_1) \times \dots \times \O(n_d - n_{d-1}) \times \O(n-n_d) \bigr)$ be represented by $Q\in \O(n)$. Its tangent space at $\lb Q \rb$  is given by
\begin{multline*}
T_{\lb Q \rb}  \Flag(n_1,\dots, n_d;n)  = \lbrace QB \in \mathbb{R}^{n \times n} : B \in \mathfrak{m}\rbrace \\
=\left\lbrace Q\begin{bsmallmatrix}
0 & B_{1,2}  &\dots & B_{1,d+1} \\
-B_{1,2}^\tp  & 0   &\dots & B_{2,d+1} \\
\vdots & \vdots & \ddots & \vdots\\
-B_{1,d+1}^\tp  & -B_{2,d+1}^\tp  & \dots & 0
\end{bsmallmatrix} \in \mathbb{R}^{n \times n} :\begin{multlined} B_{i,j}\in \mathbb{R}^{(n_i-n_{i-1}) \times (n_j-n_{j-1})}, \\ 1\le i < j \le d + 1 \end{multlined} \right\rbrace.
\end{multline*}
In particular, the dimension of a flag manifold is  given by
\[
\dim \Flag(n_1,\dots, n_d;n) = \sum_{1 \le i < j \le d+1} (n_i-n_{i-1}) (n_j - n_{j-1}).
\]
\end{proposition}
\begin{proof}
Let $M = \Flag(n_1,\dots, n_d;n) $. 
For $Q = I$, the identity matrix, this follow from $T_{\lb I \rb} M \simeq \mathfrak{g}/\mathfrak{h} \simeq \mathfrak{m}$. For $Q$ arbitrary, the left translation $L_Q: M \to M$ is a diffeomorphism, which means that $(d {L_Q})_{\lb I \rb} : T_{\lb I \rb} M \to T_{\lb Q \rb} M$ is an  isomorphism. The result then follows from $(d {L_Q})_{\lb I \rb} (X) = QX$ for all $X \in T_{\lb I \rb} M$.
\end{proof}

There are several ways to equip $\Flag(n_1,\dots, n_d;n)$ with a Riemannian metric but there is a distinguished choice that is given by a negative multiple of the \emph{Killing form} of $\mathfrak{so}(n)$, although we will not need to introduce this concept.
\begin{proposition}[Riemannian metric I]\label{prop:metric}
The metric $g$ on $\Flag(n_1,\dots, n_d;n)$ defined by
\begin{equation}\label{eq:metric1}
g_{\lb Q \rb} (X,Y) = \frac{1}{2} \tr(X^\tp  Y)
\end{equation}
for all $X,Y\in T_{\lb Q \rb} \Flag(n_1,\dots, n_d;n)$ is an $\SO(n)$-invariant metric. If we write 
\[
X = Q\begin{bsmallmatrix}
0 & B_{1,2}  &\dots & B_{1,d+1} \\
-B_{1,2}^\tp  & 0   &\dots & B_{2,d+1} \\
\vdots & \vdots & \ddots & \vdots\\
-B_{1,d+1}^\tp  & -B_{2,d+1}^\tp  & \dots & 0
\end{bsmallmatrix}, \quad 
Y = Q\begin{bsmallmatrix}
0 & C_{1,2}  &\dots & C_{1,d+1} \\
-C_{1,2}^\tp  & 0   &\dots & C_{2,d+1} \\
\vdots & \vdots & \ddots & \vdots\\
-C_{1,d+1}^\tp  & -C_{2,d+1}^\tp  & \dots & 0
\end{bsmallmatrix}  \in \mathbb{R}^{n \times n},
\]
where $B_{ij}, C_{ij}\in \mathbb{R}^{(n_i-n_{i-1}) \times (n_j-n_{j-1})}$, $1\le i < j \le d$, then $g$ may be expressed as
\begin{equation}\label{eq:metric2}
g_{\lb Q \rb}  (X, Y) = \sum_{1\le i < j \le d+1} \tr (B_{ij}^\tp  C_{ij}).
\end{equation}
\end{proposition}
\begin{proof}
We will first need to establish an  $\Ad_{\SO(n)}$-invariant inner product on $\mathfrak{so}(n)$.  It is a standard fact \cite{O'Neil1983} that bi-invariant metrics on a Lie group $G$ are in one-to-one correspondence with $\Ad_G$-invariant inner products on its Lie algebra $\mathfrak{g}$. In our case, $G = \SO(n)$, $\mathfrak{g} =\mathfrak{so}(n)$, and $\Ad_{\SO(n)} : \SO(n) \to \GL(\mathfrak{so}(n))$. Since $\SO(n)$ is compact,  by Proposition~\ref{prop:nhs} it has a bi-invariant metric, which corresponds to an $\Ad_{\SO(n)}$-invariant inner product on $\mathfrak{so}(n)$.

When $n \ne 2,4$, $\mathfrak{so}(n)$ is a simple Lie algebra  and so the $\Ad_{\SO(n)}$-invariant inner product is unique up to a scalar multiple. When $n  = 2$, $\SO(2)$ is one-dimensional and thus abelian, so the bi-invariant metric on $\SO(2)$ is unique up to a scalar. When $n  = 4$, $\SO(4) \simeq \SO(3) \times \SO(3)$ as Lie groups, so it has a two-dimensional family of bi-invariant metrics. For all values of $n$, we may take our $\Ad_{\SO(n)}$-invariant inner product (the choice is unique for all $n \ne 4$) as
\begin{equation}\label{eq:ip1}
\langle X,Y\rangle \coloneqq \frac{1}{2} \tr (X^\tp  Y)
\end{equation}
for all  $X,Y\in \mathfrak{so}(n)$.

Let $G =\SO(n)$ and  $H = \S\bigl(\O(n_1) \times \O(n_2 - n_1) \times \dots \times \O(n-n_d) \bigr)$. We will use the characterization of a flag manifold in \eqref{eq:homo2}, i.e., $\Flag(n_1,\dots, n_d;n) = G/ H$.
Since $\mathfrak{m}$ is a subspace of $\mathfrak{so}(n)$, the restriction of $\langle\cdot,\cdot\rangle$ in \eqref{eq:ip1} to $\mathfrak{m}$, denoted by $\langle \cdot,\cdot\rangle_{\mathfrak{m}}$, is an inner product on $\mathfrak{m}$. It is easy to verify that $\langle \cdot,\cdot\rangle_{\mathfrak{m}}$ is $\Ad_H$-invariant. Taken together with Propositions~\ref{prop:metric inner product correspondence}, \ref{prop:nhs}, and \ref{prop:reductive complement}, we have that $\langle\cdot,\cdot\rangle_{\mathfrak{m}}$ uniquely determines a $G$-invariant metric $g$ on $G/H$, as required.
\end{proof}

Unsurprisingly  the metric $g$ in  Proposition~\ref{prop:metric} coincides with the canonical metric on Grassmannian  ($d= 1$) introduced in \cite{EAS}. It also follows from Theorem~\ref{thm:GO} that,  with this metric $g$, $\Flag(n_1,\dots, n_d;n)$ is not merely a Riemannian manifold but also a geodesic orbit space. In fact, $g$ is the only choice of a metric that makes $\Flag(n_1,\dots,n_d;n)$ into a geodesic orbit space \cite{AA2007}. We will next derive explicit analytic expressions for geodesic (Propositions~\ref{prop:geodesics1} and \ref{prop:geodesics2}), arclength (Corollary~\ref{cor:arc-length}), geodesic distance (Proposition~\ref{prop:geodist}), and parallel transport (Proposition~\ref{prop:pt}).
\begin{proposition}[Geodesic I]\label{prop:geodesics1}
Let $\lb Q \rb \in \Flag(n_1,\dots, n_d;n) = \O(n)/\bigl(\O(n_1)\times \dots \times \O(n-n_d)\bigr)$ and $g$ be the metric in \eqref{eq:metric2}. Every geodesic on  $\Flag(n_1,\dots, n_d;n)$ passing through $\lb Q \rb$ takes the form 
\[
\lb Q(t) \rb = \left\lbrace Q\exp (tB) \begin{bsmallmatrix}
Q_1 & 0 & \dots & 0\\
0 & Q_2 & \dots & 0\\
\vdots & \vdots & \ddots & \vdots \\
0 & 0 & \dots & Q_{d+1}
\end{bsmallmatrix}  \in \O(n): Q_i \in \O(n_i- n_{i-1}),\; i=1,\dots, d+1
\right\rbrace,
\]
for some direction
\begin{equation}\label{eq:exp map direction}
B = \begin{bsmallmatrix}
0 & B_{1,2}  &\dots & B_{1,d+1} \\
-B_{1,2}^\tp  & 0   &\dots & B_{2,d+1} \\
\vdots & \vdots & \ddots & \vdots\\
-B_{1,d+1}^\tp  & -B_{2,d+1}^\tp  & \dots & 0
\end{bsmallmatrix} \in \mathbb{R}^{n \times n}, \quad \begin{multlined} B_{ij}\in \mathbb{R}^{(n_i-n_{i-1}) \times (n_j-n_{j-1})}, \\ 1\le i < j \le d+1\end{multlined}.
\end{equation}
\end{proposition}
\begin{proof}
Since $\Flag(n_1,\dots, n_d;n)$ with the metric $g$ in  Proposition~\ref{prop:metric} is a geodesic orbit space, the result follows immediately from Theorem~\ref{thm:1-parameter subgp}.
\end{proof}

\begin{corollary}[Arclength I]\label{cor:arc-length}
The arclength of a geodesic $\gamma(t) = \lb Q(t) \rb$ passing through $Q$ in the direction $B$  is given by
\[
\lVert\gamma(t)  \rVert  = t \Bigl[  \sum\nolimits_{1\le i<j \le d+1} \tr (B_{ij}^\tp  B_{ij}) \Bigr]^{1/2} 
 = t \sqrt{\frac{\tr(B^\tp  B)}{2}},
\]
where $B$ is as in \eqref{eq:exp map direction}.
\end{corollary}
\begin{proof}
This follows from the definition of arclength
\[
\lVert\gamma(t)  \rVert \coloneqq  \int_{0}^t \sqrt{g_{\lb Q \rb} (\gamma'(x), \gamma'(x)) } \; dx
\]
and the expressions for $g$ in \eqref{eq:metric1} and \eqref{eq:metric2}.
\end{proof}

\begin{proposition}[Geodesic II]\label{prop:geodesics2}
Let $\gamma$ be a geodesic in $\Flag(n_1,\dots, n_d;n) = \O(n)/\bigl(\O(n_1)\times \dots \times \O(n-n_d)\bigr)$ with $\gamma(0) = \lb Q \rb$ for some $Q\in \O(n)$ and $\gamma'(0) = H \in T_{\lb Q \rb}  \Flag(n_1,\dots, n_d;n)$. Let $Q^\tp H =  V D V^\tp $ with $V\in \O(n)$ and 
\begin{align}\label{eq:geodesic variable D}
D = \operatorname{diag} \biggl( \begin{bmatrix}
0 & -\lambda_1 \\
\lambda_1 & 0 
\end{bmatrix}, \dots, \begin{bmatrix}
0 & -\lambda_r \\
\lambda_r & 0 
\end{bmatrix}, 0_{n-2r}
\biggr)\in \mathfrak{so}(n),
\end{align}
where $2r = \rank(Q^\tp H )$ and $\lambda_1,\dots,\lambda_r$ are positive real numbers. Then $\gamma(t)  = \lb U \Sigma(t) V^\tp \rb$ where $U = QV\in \O(n)$ and 
\begin{align} \label{eq:geodesic variable sigma}
\Sigma(t) = \operatorname{diag} \biggl(
\begin{bmatrix}
\cos t\lambda_1 & -\sin t\lambda_1 \\
\sin t\lambda_1 & \cos t\lambda_1 
\end{bmatrix}, 
\dots, 
\begin{bmatrix}
\cos t\lambda_r & -\sin t\lambda_r \\
\sin t\lambda_r & \cos t\lambda_r 
\end{bmatrix}, 
I_{n - 2r}
\biggr) \in \O(n).
\end{align}
\end{proposition}
\begin{proof}
By Proposition~\ref{prop:geodesics1}, the geodesic $\gamma$ takes the form $\gamma(t) = \lb Q \exp(tB) \rb$ for some $B\in \mathfrak{so}(n)$ and $Q \in \O(n)$ representing $\gamma(0)$. Hence we have $H = \gamma'(0) = QB $ and $Q^\tp  H = B$. Since $B$ is a skew-symmetric and thus a normal matrix, by the spectral theorem \cite[Theorem~7.25]{Sheldon1997}, $B = V D V^\tp $ for some $V\in \O(n)$ and $D$ of the form in \eqref{eq:geodesic variable D}, with $2r = \rank(B) =  \rank(Q^\tp H) $ and $\lambda_1,\dots,\lambda_r$ are positive reals as they are singular values of $B$. Therefore,
\[
Q \exp(tB) = U \Sigma(t) V^\tp ,
\]
where $U = QV$ and $\Sigma(t)$ is as in \eqref{eq:geodesic variable sigma}.
\end{proof}

\begin{proposition}[Geodesic distance]\label{prop:geodist}
The geodesic distance with respect to the metric $g$ between $\lb P \rb, \lb Q \rb \in  \Flag(n_1,\dots, n_d;n) = \O(n)/\bigl(\O(n_1)\times \dots \times \O(n-n_d)\bigr)$ is
\begin{equation}\label{eq:geodist}
d(\lb P \rb,\lb Q \rb) = \sqrt{\sum\nolimits_{i=1}^r \lambda_i^2},
\end{equation}
where $\lambda_1,\dots, \lambda_r$ are positive real numbers such that $P^\tp  Q = V \Sigma V^\tp$ with $V \in \O(n)$ and 
\[
\Sigma = \operatorname{diag} \biggl(
\begin{bmatrix}
\cos \lambda_1 & -\sin \lambda_1 \\
\sin \lambda_1 & \cos \lambda_1 
\end{bmatrix}, 
\dots, 
\begin{bmatrix}
\cos \lambda_r & -\sin \lambda_r \\
\sin \lambda_r & \cos \lambda_r 
\end{bmatrix}, 
0_{n-2r}
\biggr).
\]
\end{proposition}
\begin{proof}
By Proposition~\ref{prop:geometric structure}\eqref{it:cls1}, we may regard $\Flag(n_1,\dots, n_d;n)$ as a closed, and therefore compact, submanifold of $\Gr(n_1,n)\times \dots \times \Gr(n_d,n)$. By Theorem~\ref{thm:Hopf--Rinow}, there is a distance minimizing geodesic $\lb P\exp(tB) \rb$ connecting $\lb P \rb$ and $\lb Q \rb$. By Corollary~\ref{cor:arc-length}, we get \eqref{eq:geodist} with $\lambda_1,\lambda_1,\dots,\lambda_r,\lambda_r$  the nonzero singular values of $B$. Lastly, by Proposition~\ref{prop:geodesics2}, we get the decomposition $P^\tp Q = V \Sigma V^\tp $ for some $V\in \O(n)$.
\end{proof}

Let $\mathfrak{m}$ be as in \eqref{eq:m}. For $B \in \mathfrak{m}$, we define a map
\begin{equation}\label{eq:phi}
\varphi_{B}: \mathfrak{m} \to \mathfrak{m}, \quad X \mapsto \frac{1}{2}[B,X]_{\mathfrak{m}} \coloneqq  \frac{1}{2} \operatorname{proj}_{\mathfrak{m}} ([B,X]),
\end{equation}
where $\operatorname{proj}_{\mathfrak{m}}:\mathfrak{so}(n) \to \mathfrak{m}$ is the projection from $\mathfrak{so}(n) = \mathfrak{h} \oplus \mathfrak{m}$ to $\mathfrak{m}$. For example, if $d = 2$ and 
\[
B = \begin{bmatrix}
0 & B_{12} & B_{13} \\
-B_{12}^\tp  & 0 & B_{23} \\
-B_{13}^\tp  & -B_{23}^\tp  & 0
\end{bmatrix}\in\mathfrak{m},
\quad X = \begin{bmatrix}
0 & X_{12} & X_{13} \\
-X_{12}^\tp  & 0 & X_{23} \\
-X_{13}^\tp  & -X_{23}^\tp  & 0
\end{bmatrix}\in \mathfrak{m},
\]
where $B_{ij},X_{ij}\in \mathbb{R}^{(n_i-n_{i-1}) \times (n_j-n_{j-1})}, 1\le i < j \le 3$, then 
\[
\varphi_B(X) =  \begin{bmatrix}
0 & -B_{12} X_{23}^\tp  +  X_{12} B_{23}^\tp  & B_{11}X_{23} - X_{11}B_{23} \\
X_{23}B_{12}^\tp  -  B_{23} X_{12}^\tp  & 0 & -B_{11}X_{12}^\tp  + X_{11}B_{12}^\tp \\
-X_{23}^\tp  B_{11}^\tp  + B_{23}^\tp 
X_{11}^\tp  & X_{12} B_{11}^\tp  - B_{12} X_{11}^\tp   & 0
\end{bmatrix}\in \mathfrak{m}.
\]
\begin{proposition}[Parallel transport I]\label{prop:pt}
Let $B,X \in T_{\lb I \rb} \Flag(n_1,\dots, n_d; n) \cong  \mathfrak{m}$ and $\lb Q \rb \in \Flag(n_1,\dots, n_d; n)$. The parallel transport of $QX\in T_{\lb Q \rb}  \Flag(n_1,\dots, n_d; n)$ along the geodesic $\lb Q\exp(tB) \rb$ is
\begin{equation}\label{eq:pt}
X(t) = Q \exp(tB)  e^{-\varphi_{tB}} ( X ),
\end{equation}
where $e^{-\varphi_B}  : \mathfrak{m} \to \mathfrak{m}$, for $\varphi_B$ as in \eqref{eq:phi}, is defined by
\begin{equation}\label{eq:expphi}
e^{-\varphi_B} = \sum_{k=0}^\infty \frac{(-1)^k}{k!} \varphi_B^k.
\end{equation}
\end{proposition}
\begin{proof}
This follows from applying \cite[Lemma~3.1]{Tojo1996} to $\Flag(n_1,\dots, n_d; n)$.
\end{proof}

For the $d= 1$ case, i.e., $\Flag(k;n) = \Gr(k,n)$, it is straightforward to verify that $[B,X]_\mathfrak{m} = 0$ for all $B,X\in \mathfrak{m}$. So the expression for parallel transport in \eqref{eq:pt} reduces to $X(t) = Q \exp(tB) X$, which is the well-known expression for parallel transport on the Grassmannian \cite{EAS}.

\subsection{Stiefel coordinates for the flag manifold}\label{sec:stief}

We next discuss the characterization of a flag manifold as a quotient of the Stiefel manifold \eqref{eq:homo3} and discuss  its consequences. This characterization will give our coordinates of choice for use in our optimization algorithms  (see Section~\ref{sec:gH}).
\begin{proposition}\label{prop:Stiefel}
Let $0 < n_1< \dots < n_d < n$ be  $d$ positive integers. The flag manifold $\Flag(n_1,\dots, n_d;n)$ is diffeomorphic to the homogeneous space
\begin{equation}\label{prop:Stiefel:eqn1}
\V(n_d,n) / \bigl(\O(n_1) \times \O(n_2-n_1) \times \dots \times \O(n_d- n_{d-1})\bigr)
\end{equation}
where $\V(n_d,n)$ is the Stiefel manifold of orthonormal $n_d$-frames in $n$. 
\end{proposition}
\begin{proof}
This follows from the standard characterization of $\V(n_d,n)$ is a homogeneous space of $O(n)$,
$\V(n_d,n) \cong  \O(n)/\O(n-n_d)$,  together with \eqref{eq:homo1}.
\end{proof}
For the rest of this article, we will regard the Stiefel manifold $\V(k,n)$ as the set of all $n\times k$ matrices whose column vectors are orthonormal. With this identification, Proposition~\ref{prop:Stiefel} allows us to represent a flag $\{ \mathbb{V}_i\}_{i=1}^d \in \Flag(n_1,\dots, n_d;n)$ by a  matrix $Y =[y_1,\dots,y_{n_d}] \in \mathbb{R}^{n\times n_d}$ with orthonormal $y_1,\dots,y_{n_d} \in \mathbb{R}^n$ and where the first $n_i$ of them span the subspace $\mathbb{V}_i$, $i=1,\dots, d$. This representation is not unique but if $Y'\in \mathbb{R}^{n\times n_d} $ is another such matrix, then 
\begin{equation}\label{eq:rep}
Y' = Y \begin{bsmallmatrix}
Q_1 & 0 & \dots & 0\\
0 & Q_2 &\dots & 0 \\
\vdots & \vdots & \ddots & \vdots \\
0 & 0 &\dots & Q_{d}
\end{bsmallmatrix},
\quad Q_i \in \O(n_i - n_{i-1}), \quad i =1,\dots, d.
\end{equation}
Hence $\{ \mathbb{V}_i\}_{i=1}^d \in \Flag(n_1,\dots,n_d;n)$ may be represented by the equivalence class of matrices
\begin{equation}\label{eq:equiv2}
\lb Y \rb = \left\lbrace
Y \begin{bsmallmatrix}
Q_1 & 0 & \dots & 0\\
0 & Q_2 &\dots & 0 \\
\vdots & \vdots & \ddots & \vdots \\
0 & 0 &\dots & Q_d
\end{bsmallmatrix} \in \mathbb{R}^{n \times n_d}:
\begin{multlined}
Y\in \V(n_d,n), \;
\spn\{y_1,\dots,y_{n_i}\} =  \mathbb{V}_i, \\
Q_i \in \O(n_i - n_{i-1}), \; i =1,\dots, d
\end{multlined}
\right\rbrace.
\end{equation}
We will call such a representation \emph{Stiefel coordinates} for the flag manifold.

In the following, for any $k < n$, we write
\[
I_{n,k} \coloneqq \begin{bmatrix}
I_{k} \\
0 
\end{bmatrix} \in \mathbb{R}^{n\times k},
\]
i.e., the $n \times k$ matrix comprising the first $k$ columns of the $n \times n$ identity matrix $I_n$. Thus for any $A =[a_1,\dots, a_n] \in \mathbb{R}^{n \times n}$, $AI_{n,k} = [a_1,\dots,a_k] \in  \mathbb{R}^{n \times k}$ gives us the first $k$ columns of $A$.

For a flag $\{ \mathbb{V}_i\}_{i=1}^d$, it is easy to convert between its orthogonal coordinates, i.e., $\lb Q \rb$ in \eqref{eq:equivalence class} with  $Q\in \O(n)$, and its Stiefel coordinates, i.e., $\lb Y \rb$ in \eqref{eq:equiv2} with  $Y \in \V(n_d,n)$. Given $Q \in \O(n)$, one just takes its first $n_d$ columns  to get $Y = QI_{n,n_d}$; note that $QI_{n,n_i}$ is automatically an orthonormal basis for the subspace $\mathbb{V}_i$, $i=1,\dots, d$. Given $Y \in \V(n_d,n)$, take any orthonormal basis $Y^\perp \in \V(n-n_d, n)$ of the orthogonal complement of $\im(Y)$  to get $Q = [Y, Y^\perp]\in \O(n)$. 

We now derive expressions for tangent space, metric, arclength, geodesic, and parallel transport in  Stiefel coordinates.
\begin{proposition}[Tangent space II]\label{prop:tangent2a}
Let $\lb Y \rb\in \Flag(n_1,\dots, n_d;n) = \V(n_d,n) / \bigl(\O(n_1) \times \O(n_2-n_1) \times \dots \times \O(n_d- n_{d-1})\bigr)$ be represented by $Y\in \V(n_d, n)$. Its tangent space at $\lb Y \rb$  is given by
\[
T_{\lb Y \rb}\Flag(n_1,\dots,n_d;n) = \lbrace  [Y,Y^{\perp}] B I_{n,n_d} \in \mathbb{R}^{n \times n_d}: B\in \mathfrak{m}\rbrace,
\]
where $Y^{\perp} \in \V(n-n_d,n) $ is such that $[Y,Y^{\perp}] \in \O(n)$ and $\mathfrak{m}$ is as in \eqref{eq:m}.
\end{proposition}
\begin{proof}
The  calculation is straightforward and details can be found in \cite{NSP2006}. Essentially it follows from differentiating a curve $\tau(t)$ in $\Flag(n_1,\dots,n_d;n)$ with $\tau(0) = \lb Y \rb$ and noting that the tangent vector $\tau'(0)$ is perpendicular to $\mathfrak{h}$ in \eqref{eq:h}, whose orthogonal complement is precisely $\mathfrak{m}$.
\end{proof}
The description of $T_{\lb Y \rb}\Flag(n_1,\dots,n_d;n) $ in Proposition~\ref{prop:tangent2a} is a parametric one (like the description of the unit circle as $\{(\cos \theta, \sin \theta) : \theta \in [0,2\pi)\}$). We may also derive an  implicit description of $T_{\lb Y \rb}\Flag(n_1,\dots,n_d;n) $  (like the description of the unit circle as $\{(x,y) : x^2 + y^2 = 1\}$).
\begin{corollary}[Tangent space III]\label{prop:tangent2b}
Let $\lb Y \rb\in \Flag(n_1,\dots, n_d;n) = \V(n_d,n) / \bigl(\O(n_1) \times \O(n_2-n_1) \times \dots \times \O(n_d- n_{d-1})\bigr)$ be represented by $Y\in \V(n_d, n)$. 
Let $Y$ be partition as
\[
Y = [Y_1,\dots, Y_d], \quad Y_i \in \V(n_i - n_{i-1},n), \quad i=1,\dots, d.
\]
Then its tangent space  at $\lb Y \rb$  is given by
\begin{multline}\label{eq:tangent space}
T_{\lb Y \rb}\Flag(n_1,\dots,n_d;n) = \lbrace  [X_1,\dots,X_d]\in \mathbb{R}^{n\times n_d} : X_i\in \mathbb{R}^{n \times (n_i - n_{i-1})}, \\
Y_i^\tp  X_j + X_i^\tp  Y_j = 0,\; Y_i^\tp  X_i = 0,\; 1\le i, j \le d  \rbrace.
\end{multline}
Equivalently, the matrix $[X_1,\dots, X_d]$ can be  expressed as
\[
[X_1,\dots, X_d] = [Y_1,\dots,Y_d,Y^\perp] \begin{bsmallmatrix}
0 & B_{1,2} & \dots & B_{1,d} \\
-B_{1,2}^\tp  & 0 & \dots & B_{2,d} \\
\vdots & \dots & \ddots & \vdots \\
-B_{1,d}^\tp  & -B_{2,d}^\tp  & \dots & 0  \\
-B_{1,d+1}^\tp  & -B_{2,d+1}^\tp  & \dots & -B_{d,d+1}^\tp   
\end{bsmallmatrix},
\]
where $Y^{\perp} \in \V(n-n_d,n) $ is such that $[Y,Y^{\perp}] \in \O(n)$ and $B_{ij}\in \mathbb{R}^{(n_i-n_{i-1})\times (n_j-n_{j-1})}$, $1\le i<j\le d+1$.
\end{corollary}
\begin{proof}
Since $ [Y_1,\dots, Y_d] \in \V(n_d,n)$ and $ Y_i \in \V(n_i - n_{i-1},n)$, the $Y_i$'s are characterized by 
\begin{equation}\label{eq:tangent space2}
Y_i^\tp Y_i = I_{n_i - n_{i-1}},\quad Y_i^\tp Y_j = 0,\quad i\ne  j =1,\dots, d.
\end{equation}
Differentiating \eqref{eq:tangent space2} gives us the first relation in \eqref{eq:tangent space}. On the other hand, by Proposition~\ref{prop:tangent2a} we notice that a tangent vector in $T_{\lb Y \rb}\Flag(n_1,\dots,n_d;n)$ is written as $[Y,Y^\perp] B I_{n,n_d}$ for some $B\in \mathfrak{m}$, from which we may easily verify the second relation in \eqref{eq:tangent space}.
\end{proof}

Comparing Propositions~\ref{prop:tangent1} and \ref{prop:tangent2a}, for a tangent vector $QB \in T_{\lb Q \rb} \Flag(n_1,\dots,n_d;n)$ in orthogonal coordinates $Q \in \O(n)$, its corresponding tangent vector in Stiefel coordinates $Y = Q I_{n,n_d} \in \V(n_d,n)$ is simply given by $QB I_{n,n_d} \in T_{\lb Y \rb} \Flag(n_1,\dots,n_d;n)$. Conversely, $[Y,Y^{\perp}] B I_{n,n_d} \in T_{\lb Y \rb} \Flag(n_1,\dots,n_d;n)$ in Stiefel coordinates corresponds to $QB \in T_{\lb Q \rb} \Flag(n_1,\dots,n_d;n)$ in orthogonal coordinates where $Q = [Y,Y^{\perp}]$. Note that from the matrix $B I_{n,n_d}$, i.e., just the first $n_d$ columns of $B \in \mathfrak{m}$, the full matrix $B$ can be easily and uniquely recovered by its skew symmetry.

The straightforward translation between orthogonal and Stiefel coordinate representations of points and tangent vectors on a flag manifold allows us to immediately deduce analogues of Propositions~\ref{prop:metric}, \ref{prop:geodesics1}, \ref{prop:pt}, and Corollary~\ref{cor:arc-length}.

\begin{proposition}[Riemannian metric II]\label{prop:metric2}
The metric $g$ at a point $\lb Y \rb \in \Flag(n_1,\dots,n_d;n) = \V(n_d,n) / \bigl(\O(n_1) \times \O(n_2-n_1) \times \dots \times \O(n_d- n_{d-1})\bigr)$ is given by 
\begin{equation}\label{eq:g2}
g_{\lb Y \rb}(W,Z) = \sum_{1\le i < j \le d+1} \tr(B_{ij}^\tp C_{ij}),
\end{equation}
where $W,Z\in T_{\lb Y \rb} \Flag(n_1,\dots,n_d;n)$ are
\[
W = [Y,Y^\perp] \begin{bsmallmatrix}
0 & B_{1,2}  &\dots  &B_{1,d} \\
-B_{1,2}^\tp  & 0   &\dots& B_{2,d} \\
\vdots & \vdots & \ddots  &\vdots\\
-B_{1,d}^\tp & -B_{2,d}^\tp  & \cdots  & 0\\
-B_{1,d+1}^\tp  & -B_{2,d+1}^\tp  &  \dots & -B_{d,d+1}^\tp
\end{bsmallmatrix},\quad  Z = [Y,Y^\perp]\begin{bsmallmatrix}
0 & C_{1,2}  &\dots  &C_{1,d} \\
-C_{1,2}^\tp  & 0   &\dots& C_{2,d} \\
\vdots & \vdots & \ddots  &\vdots\\
-C_{1,d}^\tp & -C_{2,d}^\tp  & \cdots  & 0\\
-C_{1,d+1}^\tp  & -C_{2,d+1}^\tp  &  \dots & -C_{d,d+1}^\tp
\end{bsmallmatrix} \in \mathbb{R}^{n \times n_d}.
\]
\end{proposition}

\begin{proposition}[Arclength II, Geodesics III]\label{prop:geodesic3a}
Let $\lb Y \rb \in \Flag(n_1,\dots,n_d;n)  = \V(n_d,n) / \bigl(\O(n_1) \times \O(n_2-n_1) \times \dots \times \O(n_d- n_{d-1})\bigr)$ and $g$ be the metric in \eqref{eq:g2}. Every geodesic $\gamma$ on  $\Flag(n_1,\dots, n_d;n)$ passing through $\lb Y \rb$ takes the form 
\[
\gamma(t) = \lb Y(t) \rb  = \left\lbrace
[Y,Y^\perp] \exp (tB) \begin{bsmallmatrix}
Q_1 & 0 & \dots & 0\\
0 & Q_2 &\dots & 0 \\
\vdots & \vdots & \ddots & \vdots \\
0 & 0 &\dots & Q_d\\
0 & 0 & \dots & 0
\end{bsmallmatrix} \in  \V(n_d,n) : \begin{multlined} Q_i\in \O(n_i - n_{i-1}),\\ i=1,\dots, d \end{multlined}
\right\rbrace,
\]
where $[Y,Y^\perp]\in \O(n)$ and $B\in \mathfrak{m}$. In particular, the arclength of $\gamma(t)$ is  
\[
\lVert \gamma(t) \rVert =  t \Bigl[  \sum\nolimits_{1\le i<j \le d+1} \tr (B_{ij}^\tp  B_{ij}) \Bigr]^{1/2}. 
\]
\end{proposition}

\begin{proposition}[Parallel transport II]\label{prop:pt2}
Let $\lb Y \rb \in \Flag(n_1,\dots,n_d;n)$ and
\[
[Y,Y^{\perp}] B I_{n,n_d}, \; [Y,Y^{\perp}] X I_{n,n_d} \in T_{\lb Y \rb} \Flag(n_1,\dots,n_d;n).
\]
The parallel transport of $[Y,Y^{\perp}] X I_{n,n_d}$ along the geodesic $\lb [Y,Y^\perp] \exp (tB)I_{n,n_d}  \rb $ is given by
\begin{equation}\label{eq:pt2}
X(t) = [Y,Y^\perp]  \exp(tB)  e^{-\varphi_{tB}} ( X ) I_{n,n_d},
\end{equation}
with $e^{-\varphi_{tB}}$ defined as in \eqref{eq:phi} and \eqref{eq:expphi}.
\end{proposition}
While is also straightforward to obtain analogues of Proposition~\ref{prop:geodesics2} and \ref{prop:geodist} in Stiefel coordinates, we omit them as the expressions are more involved and we will not need them in the sequel.

\section{Flag manifolds as matrix manifolds}\label{sec:mfld}

By Proposition~\ref{prop:geometric structure}\eqref{it:cls1} and \eqref{it:cls2}, we see that a flag manifold may be regarded as a submanifold of a product of Grassmannians. Since a Grassmannian can be represented as a subset of matrices in $\mathbb{R}^{n\times n}$ \cite[Example~1.2.20]{Nicolaescu},
\begin{equation}\label{eq:P}
\Gr(k,n) \cong \{P\in \mathbb{R}^{n\times n}: P^2 = P = P^\tp,\; \tr(P) = k\},
\end{equation}
so can a flag manifold; and we will discuss two different ways do this, corresponding to \eqref{it:cls1} and \eqref{it:cls2} in Proposition~\ref{prop:geometric structure}:
\begin{align*}
\Flag(n_1,\dots,n_d;n) &\subseteq \Gr(n_1,n)\times  \Gr(n_2,n) \times \dots \times \Gr(n_d,n),\\
\Flag(n_1,\dots,n_d;n) &\subseteq \Gr(n_1,n)\times  \Gr(n_2-n_1,n) \times \dots \times \Gr(n_d-n_{d-1},n).
\end{align*}

The correspondence in \eqref{eq:P} is given by a map that takes a $k$-dimensional subspace $\mathbb{W}\in \Gr(k,n)$ to its orthogonal projector,
\begin{equation}\label{eq:epsGr}
\varepsilon: \Gr(k,n) \to \mathbb{R}^{n\times n}, \quad \mathbb{W} \mapsto W W^\tp,
\end{equation}
where $W \in \mathbb{R}^{n \times k}$ is any orthonormal basis of $\mathbb{W}$. Note that if $W'$ is another such $n \times k$ matrix, then $W' = W Q$ for some $Q\in \O(k)$ and so $W'W^{\prime\tp} =WW^\tp$ and the map $\varepsilon$ is well-defined. It is also injective and its image is precisely the set on the right of \eqref{eq:P}.

\subsection{Projection coordinates for the flag manifold}\label{sec:proj}

We will construct our first analogue of  \eqref{eq:epsGr} for the  flag manifold. Let
\begin{equation}\label{eq:eps}
\varepsilon : \Flag(n_1,\dots,n_d;n) \to \mathbb{R}^{nd\times nd}, \quad
\{\mathbb{V}_i\}_{i=1}^d \mapsto \diag (V_1V_1^\tp , \dots, V_d V_d^\tp  ),
\end{equation}
where $V_i \in \mathbb{R}^{n\times n_i}$ is an orthonormal basis of $\mathbb{V}_i$, $i=1,\dots, d$, and the image is a block-diagonal matrix in $\mathbb{R}^{nd \times nd}$ with $d$ blocks  $V_1V_1^\tp , \dots, V_d V_d^\tp \in \mathbb{R}^{n \times n}$. In fact, the map in \eqref{eq:eps} is essentially the map in \eqref{eq:j} that we used to establish Proposition~\ref{prop:geometric structure}\eqref{it:cls1} except that we identify the Grassmannians with sets of projection matrices as in \eqref{eq:P}.

\begin{proposition}\label{prop:representation of flags}
The flag manifold $\Flag(n_1,\dots, n_d;n)$  is diffeomorphic to 
\begin{equation}\label{eq:Pflag}
\{P = \diag(P_1,\dots,P_d)\in \mathbb{R}^{nd\times nd}: P_i^2 = P_i = P_i^\tp,\; \tr(P_i) = n_i, \; P_j P_i = P_i, \; i<j \}.
\end{equation}
\end{proposition}
\begin{proof}
One may check that $\varepsilon$ in \eqref{eq:eps} has its image in contained in the set \eqref{eq:Pflag}; and the map that takes $P= \diag(P_1,\dots,P_d)$ to the flag $\{\im(P_i)\}_{i=1}^d \in \Flag(n_1,\dots, n_d;n) $ is its inverse.
\end{proof}
We will call the representation in Proposition~\ref{prop:representation of flags} \emph{projection coordinates} for the flag manifold. Unlike the orthogonal and Stiefel coordinates introduced earlier, which are not unique, projection coordinates are unique.
Let $\{\mathbb{V}_i\}_{i=1}^d \in \Flag(n_1,\dots, n_d;n)$ with
\begin{enumerate}[\upshape (a)]
\item\label{it:o} orthogonal coordinates $\lb Q \rb$ for some $Q  \in \O(n)$;
\item\label{it:s} Stiefel coordiantes $\lb Y \rb$ for some $Y \in \V(n_d, n)$;
\item\label{it:p} projection coordinates  $P$ as in \eqref{eq:Pflag}.
\end{enumerate}
We have seen how we may easily convert between orthogonal and Stiefel coordinates after \eqref{eq:equiv2}, we now see how they may be interchanged with projection coordinates just as easily:
\begin{itemize}
\item[\eqref{it:o}$\to$\eqref{it:p}:] Given $Q=[q_1,\dots,q_n] \in \O(n)$, let $Q_i = [q_1,\dots, q_{n_i}] \in \V(n_i, n)$; then $P_i = Q_i Q_i^\tp $, $i = 1,\dots,d$. 

\item[\eqref{it:s}$\to$\eqref{it:p}:] Given $Y = [y_1,\dots,y_{n_d}] \in \V(n_d, n)$, let $Y_i = [y_1,\dots, y_{n_i}] \in \V(n_i, n)$; then $P_i = Y_i Y_i^\tp $, $i = 1,\dots,d$.

\item[\eqref{it:p}$\to$\eqref{it:s}:] Given $P =\diag(P_1,\dots,P_d)$, let  $y_1,\dots, y_{n_{i}}$ be an orthonormal basis of $\im(P_i)$; then $Y_i= [y_1,\dots, y_{n_i}] \in \V(n_i, n)$, $i=1,\dots, d$.

\item[\eqref{it:p}$\to$\eqref{it:o}:]  As above but appending an orthonormal basis  $y_{n_d+1},\dots, y_{n}$ of $\im(P_d)^\perp$ gives us $Q = [y_1,\dots, y_{n_d}, y_{n_d+1},\dots, y_{n}] \in \O(n)$.
\end{itemize}

As is the case for the Grassmannian, the flag manifold has several extrinsic coordinates systems with which differential geometric objects and operations have closed-form analytic expressions and  where one coordinate representation can be transformed to another with relative ease. This flexibility to switch between coordinate systems can be exploited in computations but as we will see next, it can also be exploited in deriving the requisite analytic expressions.

\begin{proposition}[Tangent spaces IV]\label{prop:tangent2}
Let $P = \diag(P_1,\dots,P_d) \in \Flag(n_1,\dots, n_d;n)$ as represented in \eqref{eq:Pflag}.
Then the tangent space is given by
\begin{multline}\label{eq:tangent2}
T_P\Flag(n_1,\dots,n_d;n) = \lbrace Z = \diag(Z_1,\dots,Z_d) \in \mathbb{R}^{nd\times nd} :
Z_i P_i + P_i Z_i= Z_i = Z_i^\tp,\\  \tr(Z_i) = 0, \; Z_j P_i + P_j Z_i = Z_i,\; i<j, \; i,j=1,\dots, d. \rbrace
\end{multline}
\end{proposition}
\begin{proof}
Let $\gamma(t)$ be a curve in $\Flag(n_1,\dots,n_d;n)$ as characterized by \eqref{eq:Pflag}, i.e., $\gamma : (-1, 1) \to \mathbb{R}^{nd\times nd}$, $t \mapsto  \diag\bigl(P_1(t),\dots, P_d(t)\bigr)$ where
\begin{equation}\label{eq: curve in flag1}
P_i(t)^2 = P_i(t), \; P_i(t)^\tp  = P_i(t),\; \tr(P_i(t)) = n_i,\; P_j(t)P_i(t) = P_i(t),\; i<j, \; i,j=1,\dots,d,
\end{equation}
for all $t\in (-1,1)$. Taking derivatives of these relations at $t = 0$ gives the required description.
\end{proof}

Again the ease of translation from orthogonal and Stiefel coordinates to projection coordinates yields counterparts of Proposition~\ref{prop:metric}--\ref{prop:pt} readily. We will just provide expressions for geodesic and parallel transport as examples.

\begin{proposition}[Geodesics IV]\label{prop:geodesics3}
Let $P = \diag(P_1,\dots,P_d) \in \Flag(n_1,\dots, n_d;n)$ be as represented in \eqref{eq:Pflag} and $Z = \diag(Z_1,\dots,Z_d)\in T_P\Flag(n_1,\dots, n_d;n)$ be as represented in  \eqref{eq:tangent2}. Then there exist  $Y \in \V(n_d, n)$ and skew-symmetric $B \in \mathbb{R}^{n \times n}$ such that for $Y_i = YI_{n,n_i}$, $B_i =BI_{n,n_i} \in \mathbb{R}^{n \times n_i}$,
\begin{equation}\label{eq:pyb}
P_i = Y_i Y_i^\tp, \quad Z_i = Y_i B_i^\tp + B_iY_i^\tp, \quad i=1,\dots,d;
\end{equation}
and a geodesic $P(t)$ passing through $P$ in the direction $Z$  takes the form
\begin{equation}\label{eq:geo3}
\{\diag \bigl(P_1(t),\dots, P_d(t)\bigr) \in \mathbb{R}^{nd \times nd} : P_i(t) = Y_i(t) Y_i(t)^\tp, \; Y_i(t) = [Y,Y^\perp] \exp(tB)I_{n,n_i}\}.
\end{equation}
\end{proposition}
\begin{proof}
The matrix $Y$ is just $P$ in Stiefel coordinates and may be obtained from \eqref{it:p}$\to$\eqref{it:s} above.  By Proposition~\ref{prop:geodesic3a}, in Stiefel coordinates, the geodesic through $\lb Y \rb$ in  direction $[Y,Y^{\perp}] B I_{n,n_d}$ is
\[
\lb [Y,Y^{\perp}]  \exp(tB) I_{n,n_d} \rb.
\]
By \eqref{it:s}$\to$\eqref{it:p},  $Y$ and $P$ are related by
$P_i = Y_iY_i^\tp$, $i=1,\dots,d$, which upon differentiation gives $Z_i = Y_i B_i^\tp + B_i Y_i^\tp$. The required expression \eqref{eq:geo3} then follows.
\end{proof}
The observant reader might have noticed that $B_{d+1}$ does not appear in \eqref{eq:pyb} --- the reason is that since $B$ is skew-symmetric, $B_{d+1}$ is uniquely determined by $B_1,\dots, B_d$. 
\begin{proposition}[Parallel transport III]\label{prop:pt1}
Let $P$, $Z$, $Y$, $B$,  and $P(t)$ be as in Proposition~\ref{prop:geodesics3}. Let $Y^\perp \in \V(n-n_d, n)$ be such that $[Y, Y^\perp] \in \O(n)$ and set
\begin{equation}\label{eq:geo4}
Y_i(t)=[Y, Y^\perp] \exp(tB)I_{n,n_i}, \quad X_i(t) = [Y, Y^\perp] \exp(tB)  e^{-\varphi_{tB}} ( X ) I_{n,n_i},  \quad i=1,\dots, d.
\end{equation}
Then the parallel transport of the tangent vector $Z$ along the geodesic $P(t)$ is given by 
\begin{equation}\label{eq:pt3}
Z(t) = \diag \bigl(Z_1(t),\dots, Z_d(t)\bigr),\quad Z_i(t) = Y_i(t) X_i(t)^\tp + X_i(t) Y_i(t)^\tp,\quad i=1,\dots, d.
\end{equation}
\end{proposition}
\begin{proof}
As in the proof of Proposition~\ref{prop:geodesics3}, we obtain the corresponding projection coordinates $P = \diag(P_1,\dots, P_d)$, $P_i = Y_i Y_i^\tp$, $Y_i = YI_{n,n_i}$, $i=1,\dots, d$. Differentiating these relations give a tangent vector $Z = \diag (Z_1,\dots, Z_d)\in T_P \Flag(n_1,\dots,n_d;n)$ in projection coordinates as
\[
Z_i = Y_i X_i^\tp + X_iY_i^\tp,\quad i=1,\dots, d,
\]  
where $X\in T_{\lb Y \rb} \Flag(n_1,\dots,n_d;n)$ is the expression of the same tangent vector in Stiefel coordinates as in Proposition~\ref{prop:tangent2a} and $X_i = XI_{n,n_i}$, $i=1,\dots, d$. The required expressions \eqref{eq:geo4} and \eqref{eq:pt3} then follow from the expression \eqref{eq:pt2}  for parallel transport in terms of Stiefel coordinates $Y$.
\end{proof}

\subsection{Refined projection coordinates for the flag manifold}\label{sec:rproj}

We discuss a variation of projection coordinates on flag manifolds based on  Proposition~\ref{prop:geometric structure}\eqref{it:cls2}. As in Section~\ref{sec:proj}, if we identify the Grassmannians as sets of projection matrices as in \eqref{eq:P}, then the map in \eqref{eq:j2} becomes
\begin{equation}\label{eq:embedding of flag}
\varepsilon': \Flag(n_1,\dots,n_d;n) \to \mathbb{R}^{nd \times nd},\quad \{\mathbb{V}_i\}_{i=1}^d \mapsto \diag(W_1W_1^\tp,\dots, W_dW_d^\tp),
\end{equation}
where column vectors of $W_i\in \mathbb{R}^{n\times (n_i - n_{i-1})}$ form an orthonormal basis of $\mathbb{V}_{i-1}^{\perp}$, the orthogonal complement of $\mathbb{V}_{i-1}$ in $\mathbb{V}_i$, $i =1,\dots,d$. This gives us another description of $\Flag(n_1,\dots,n_d;n)$ as a matrix manifold, an analogue of Proposition~\ref{prop:representation of flags}.
\begin{proposition}\label{prop:reduced coord}
The flag manifold $\Flag(n_1,\dots,n_d;n)$ is diffeomorphic to
\begin{equation}\label{eq:Rflag}
\lbrace
R = \diag(R_1,\dots, R_d)\in \mathbb{R}^{nd \times nd}: R_i^2 = R_i = R_i^\tp,\; \tr (R_i) = n_i - n_{i-1},\; R_i  R_j = 0,\;  i < j
\rbrace.
\end{equation}
\end{proposition}
We call the representation in Proposition~\ref{prop:reduced coord} \emph{reduced projection coordinates} on the flag manifold $\Flag(n_1,\dots,n_d;n)$. Again, it is straightforward to translate between the other three coordinates systems and reduced projection coordinates.
This readily yields expressions for metric, tangent space, geodesic, and parallel transport in reduced projection coordinates as before. We will state those for tangent space and metric as examples.
\begin{proposition}[Tangent spaces V]\label{prop:tangent3}
Let $R = \diag(R_1,\dots,R_d) \in \Flag(n_1,\dots, n_d;n)$ be as represented in \eqref{eq:Rflag}.
Then the tangent space is given by
\begin{multline}\label{eq:tangent3}
T_R\Flag(n_1,\dots,n_d;n) = \lbrace  Z =\diag(Z_1,\dots,Z_d) \in \mathbb{R}^{nd\times nd} :
R_i Z_i + Z_i R_i = Z_i = Z_i^\tp,\\ \tr(Z_i) = 0, \; Z_i R_j + R_i Z_j =0,\;1\le i < j \le d. \rbrace
\end{multline}
\end{proposition}

Propositions~\ref{prop:reduced coord} and \ref{prop:tangent3} give an alternative way to obtain the metric  $g$ in Proposition~\ref{prop:metric}. Let $g_i$ be the standard metric on $\Gr(n_i-n_{i-1},n)$, $i=1,\dots,d$. Then it is straightforward to verify that $g$ is the pull-back of $\sum_{i=1}^d g_i$ via the embedding  \eqref{eq:j2} in Proposition~\ref{prop:geometric structure}\eqref{it:cls2}. This also gives us an expression for the metric in terms of reduced projection coordinates.
\begin{proposition}[Riemannian metric III]\label{prop:metric3}
Let $R = \diag(R_1,\dots, R_d) \in \Flag(n_1,\dots,n_d;n)$ be as in \eqref{eq:Rflag}.
Let $W= \diag(W_1,\dots,W_d),Z =\diag(Z_1,\dots,Z_d)\in T_R \Flag(n_1,\dots, n_d;n)$ be as in \eqref{eq:tangent3}. Then there exist $V_i, A_i, B_i\in \mathbb{R}^{n\times (n_i - n_{i-1})}$ such that
\begin{equation}\label{eq:VAB}
V_i V_i^\tp = R_i,\; V_i^\tp V_i = I_{n_i - n_{i-1}}, \;
V_i A_i^\tp + A_i V_i^\tp = W_i,\;
V_i^\tp A_i = 0,\;
V_i B_i^\tp + B_i V_i^\tp = Z_i,\;
V_i^\tp B_i = 0,
\end{equation}
and the metric $g$  is given by 
\[
g_R(W,Z) =  \sum_{i=1}^d   \tr(A_i^\tp B_i).
\]
\end{proposition}
\begin{proof}
As the Grassmannian is just a flag manifold with $d=1$, all our earlier discussions about Stiefel and projection coordinates also apply to it. So for $W_i,Z_i \in T_{R_i} \Gr(n_i-n_{i-1},n)$ in projection coordinates, there exist $V_i, A_i, B_i\in \mathbb{R}^{n\times (n_i - n_{i-1})}$ satisfying \eqref{eq:VAB}.
The standard Riemannian metric $g_i$ on $\Gr(n_i-n_{i-1},n)$ at $R_i$ is then  given by $g_i (W_i,Z_i) =  \tr( A_i^\tp B_i)$ and thus we have 
\[
g_R(W,Z) = \sum_{i=1}^d g_i (W_i,Z_i) = \sum_{i=1}^d   \tr(A_i^\tp B_i). \qedhere
\]
\end{proof}

\section{Riemannian Gradient and Hessian over the flag manifold}\label{sec:gH}

We will derive expressions for the Riemannian gradient and Riemannian Hessian of a real-valued function on a flag manifold, the main ingredients of optimization algorithms.  Although in principle we may use any of the four extrinsic coordinate systems introduced in the last two sections --- orthogonal (as $n \times n$ orthogonal matrices), Stiefel (as $n \times n_d$ orthonormal matrices), projection or reduced projection (as $d$-tuples of $n \times n$ projection matrices) coordinates ---  Stiefel coordinates give  the most economical representation and we will use this as our coordinates of choice. So in the following we will identify
\begin{equation}\label{eq:identify}
\Flag(n_1,\dots,n_d;n) = \V(n_d,n) / \bigl(\O(n_1) \times \O(n_2-n_1) \times \dots \times \O(n_d- n_{d-1})\bigr).
\end{equation}
Our expressions for gradient and Hessian in Stiefel coordinates may of course be converted to other coordinates --- straightforward although the results may be notationally messy.
\begin{proposition}[Riemannian gradient]\label{prop:gradient}
Let $f:\Flag(n_1,\dots,n_d;n) \to \mathbb{R}$ be a smooth function expressed in Stiefel coordinates $Y \in \V(n_d, n)$. Define the $n\times n_d$ matrix of partial derivatives,
\begin{equation}\label{eq:fY}
f_Y \coloneqq \biggl[  \frac{\partial f}{\partial y_{ij}} \biggr]_{i,j=1}^{n,n_d}.
\end{equation}
Write $Y = [Y_1,\dots, Y_d]$ where $Y_i\in \mathbb{R}^{n\times (n_i-n_{i-1})}$ and $f_Y = [f_{Y_1},\dots, f_{Y_d}]$ where $f_{Y_i}$ is the $n\times (n_i-n_{i-1})$ submatrix, $i=1,\dots, d$. Then its Riemannian gradient at $\lb Y \rb \in \Flag(n_1,\dots,n_d;n) $ is given by $\nabla f(\lb Y \rb) = [\Delta_1, \dots, \Delta_d]$ where
\begin{equation}\label{eq:gradient variable delta}
\Delta_i = f_{Y_i} - \Bigl(Y_i Y_i^\tp  f_{Y_i} + \sum\nolimits_{j\ne i } Y_j f_{Y_j}^\tp  Y_i\Bigr), \quad i=1,\dots,d.
\end{equation}
\end{proposition}
\begin{proof}
For any $X\in T_{\lb Y \rb} \Flag(n_1,\dots,n_d;n)$, let $X_a\in \mathbb{R}^{n\times n}$ be the unique skew-symmetric matrix such that $X = Q X_a I_{n,n_d}$, where $Q\in \O(n)$ is such that $Y = QI_{n,n_d}$. Since the metric expressed in Stiefel coordinates \eqref{eq:g2} and expressed in orthogonal coordinates \eqref{eq:metric1} must be equal,
\begin{equation}\label{eq:tmp1}
g_{\lb Y \rb} \bigl( \nabla f (\lb Y \rb),X \bigr)  = g_{\lb Q \rb} \bigl(\nabla f  (\lb Y \rb)_a, X_a\bigr) = \frac{1}{2} \tr ( \nabla f  (\lb Y \rb)_a^\tp X_a).
\end{equation}
By definition of Riemannian gradient \eqref{eq:gH}, we also have
\begin{equation}\label{eq:tmp2}
g_{\lb Y \rb}(\nabla f(\lb Y\rb), X) =  \tr (f_Y^\tp  X_a).
\end{equation}
Comparing \eqref{eq:tmp1} and \eqref{eq:tmp2}, we see that $\nabla f(\lb Y\rb) $ is the projection of $f_Y$ onto $T_{\lb Y \rb} \Flag(n_1,\dots,n_d;n)$, i.e., $f_Y = \nabla f(\lb Y\rb)  + Z$ for some $Z$ orthogonal to $T_{\lb Y \rb} \Flag(n_1,\dots,n_d;n)$. We may take $Z = [Z_1,\dots, Z_d]$ to be
\[
Z_i \coloneqq Y_i Y_i^\tp  f_{Y_i} + \sum\nolimits_{j\ne i } Y_j f_{Y_j}^\tp  Y_i,\quad i=1,\dots,d,
\]
and verify that because of \eqref{eq:tangent space}, we indeed have $f_Y - Z\in T_{\lb Y \rb} \Flag(n_1,\dots,n_d;n)$ and thus $Z$ is orthogonal to $T_{\lb Y \rb} \Flag(n_1,\dots,n_d;n)$.
\end{proof}
The Riemannian gradient $\nabla f$ may also be derived by solving an optimization problem as in \cite{NSP2006}. Note that if $d = 1$, \eqref{eq:gradient variable delta} becomes $\nabla f (\lb Y \rb) = \Delta = f_Y -  Y Y^\tp  f_Y$, the well-known expression for Riemannian gradient of Grassmannian in \cite{EAS}.

\begin{proposition}[Riemannian Hessian]\label{prop:Hessian}
Let $f:\Flag(n_1,\dots,n_d;n) \to \mathbb{R}$ be a smooth function expressed in Stiefel coordinates $Y \in \V(n_d, n)$ and let $f_Y$ be  as in \eqref{eq:fY}. Then its Riemannian Hessian $\nabla^2f(\lb Y \rb)$ at $\lb Y \rb \in \Flag(n_1,\dots,n_d;n) $ is the symmetric bilinear form given by
\begin{equation}\label{eq:Hessian}
\nabla^2f(\lb Y \rb) (X,X') = f_{Y,Y} ( X, X') - \frac{1}{2} \bigl[ \tr(f_Y^\tp  Q B^\tp  Q^\tp  X') +  \tr(f_Y^\tp  Q C^\tp  Q^\tp  X)  \bigr) \bigr],
\end{equation}
for $X,X'\in T_{\lb Y \rb} \Flag(n_1,\dots,n_d;n)$, where
\begin{equation}\label{eq:fYY}
f_{Y,Y} (X,X') \coloneqq \sum_{i,k=1}^n \sum_{j,l=1}^{n_d} \frac{\partial^2 f}{\partial y_{ij} \partial y_{kl}} x_{ij} x'_{kl},
\end{equation}
$Q\in \O(n)$ is such that $Q I_{n,n_d} = Y$, and $B,C \in \mathbb{R}^{n \times n}$ are the unique skew-symmetric matrices such that $X = Q B I_{n,n_d}$, $X' = Q C I_{n,n_d}$ respectively.
\end{proposition}
\begin{proof}
By Proposition~\ref{prop:geodesic3a}, a geodesic $\gamma$ with $\gamma'(0) = X$ and $\gamma(0) = \lb Y \rb$ takes the form $\gamma(t) = \lb Q\exp(tB) I_{n,n_d} \rb$  where $Q \in \O(n)$ is such that $Y = QI_{n,n_d}$ and $X = Q B I_{n,n_d}$. Applying chain rule,
\[
\frac{d}{dt} f(\gamma(t)) = \tr\bigl( f_Y^\tp  \gamma'(t)\bigr),\qquad
\frac{d^2 }{dt^2} f(\gamma(t))  = \tr\bigl(\gamma'(t)^\tp   f^\tp _{\gamma(t),\gamma(t)} \gamma'(t)\bigr) + \tr\bigl(f_Y^\tp \gamma''(t)\bigr);
\]
followed by evaluating at $t = 0$ gives
\[
\nabla^2f(\lb Y \rb)(X,X) =  \frac{d^2 }{dt^2} f(\gamma(t)) \Big|_{t=0} = f_{Y,Y} (X, X)  - \tr (f_Y^\tp  Q B^\tp  Q ^\tp  X).
\]
The required expression \eqref{eq:Hessian} then follows from \eqref{eq:gH} and \eqref{eq:polar}.
\end{proof}
If $d = 1$, \eqref{eq:Hessian} reduces to the well-known expression for Riemannian Hessian of the Grassmannian \cite[Section~2.5.4]{EAS} since
\begin{multline*}
\nabla^2f(\lb Y \rb) (X,X') = f_{Y,Y} (X, X') - \tr (f_Y^\tp   Q B^\tp  Q^\tp X') \\
 = f_{Y,Y} (X, X') - \tr (f_Y^\tp Y (X')^\tp X)  = f_{Y,Y} (X, X') - \tr (X^\tp X' Y^\tp f_Y).
\end{multline*}

There is slight inconsistency in our definitions of $f_Y$ and $f_{YY}$ to make these expressions easily portable into computer codes. To be consistent with \eqref{eq:fYY}, we could define $f_Y$ as a linear form:
\[
f_Y(X) = \sum_{i,j=1}^{n,n_d} \frac{\partial f}{\partial y_{ij}}  x_{ij}
\]
for $X \in  T_{\lb Y \rb} \Flag(n_1,\dots,n_d;n)$. Alternatively, to be consistent with \eqref{eq:fY}, we could define $f_{YY}$ to be a hypermatrix of partials (this is not a $4$-tensor, just a convenient way to represent a $2$-tensor):
\[
f_{YY} =\biggl[ \frac{\partial^2 f}{\partial y_{ij} \partial y_{kl}} \biggr]_{i,j,k,l=1}^{n,n_d, n, n_d}.
\]

\section{Optimization algorithms over flag manifolds}\label{sec:optim}

With analytic expressions for points, tangent vectors, metric, geodesic, parallel transport, Riemannian gradient and Hessian in place, Riemannian manifold optimization algorithms are straightforward to derive from the usual ones. For example, for steepest descent, instead of adding a negative multiple of the gradient to the current iterate, we move the current iterate along the geodesic with initial velocity vector given by the negative of the gradient. Again, we may do this in any of the four coordinates system we have introduced although for the same reason in Section~\ref{sec:gH}, we prefer the Stiefel coordinates. Thus here we will assume the identification \eqref{eq:identify} as before.

\subsection{Steepest descent over a flag manifold}

We describe this in Algorithm~\ref{alg:sd}. A point $\lb Y \rb \in \Flag(n_1,\dots,n_d;n)$ is represented in Stiefel coordinates, i.e., as an orthonormal matrix $Y \in \mathbb{R}^{n \times n_d}$, $Y^\tp Y=I$. As usual, $Y^{\perp} \in \mathbb{R}^{n \times (n-n_d)}$ is such that $X \coloneqq [ Y, Y^{\perp}] \in \O(n)$. The Riemannian gradient $\nabla f \in \mathbb{R}^{n \times n_d}$ is given by Proposition~\ref{prop:gradient} and we set $G = - \nabla f$ to be the search direction. 
The exponential map direction $B\in \mathbb{R}^{n\times n}$ is uniquely obtained from $B I_{n,n_d} = [Y,Y^{\perp}]^\tp G$, i.e., $B$ is the unique skew-symmetric matrix whose first $n_d$ columns is  $[Y,Y^{\perp}]^\tp G$. The next iterate is then found along the geodesic determined by the current iterate and the direction as in Proposition~\ref{prop:geodesic3a}, although the exact line search may be substituted by any reasonable strategy for choosing step size. Note in particular that Algorithm~\ref{alg:sd} does not involve parallel transport.

\begin{algorithm}
  \caption{Steepest descent in Stiefel coordinates}
  \label{alg:sd}
  \begin{algorithmic}[1]
    \Require $\lb Y_0 \rb \in \Flag(n_1,\dots,n_d;n)$ with $Y_0 \in \mathbb{R}^{n\times n_d}$ and $Y_0^\tp Y_0=I$;
    \State find $Y_0^{\perp} \in \mathbb{R}^{n \times (n-n_d)}$ such that $[Y_0,Y_0^\perp ] \in \O(n)$;
    \State set $X_{0} = [Y_0, Y_0^\perp ]$;
    \For {$i =0, 1, \dots$}
    \State set $G_{i} = -\nabla f(\lb Y_i \rb)$; \Comment{gradient at $\lb Y_{i} \rb$ as in \eqref{eq:gradient variable delta}}
    \State set $X_i = [Y_i, Y_i^\perp ]$;
    \State compute $\widehat{B} = X_{i}^\tp  G_{i}$;
    \State set $B \in \mathbb{R}^{n \times n}$ as $B_{ij} = \widehat{B}_{ij}$ for $j \leq n_d$;
    \State \phantom{set $B \in \mathbb{R}^{n \times n}$ as} $B_{ij} = -\widehat{B}_{ji}$ for $j \geq n_d$ and $i \leq n_d$;
    \State \phantom{set $B \in \mathbb{R}^{n \times n}$ as} $B_{ij} = 0$ otherwise;
    \State minimize $ f\bigl(X_{i} \exp(tB) I_{n,n_d}\bigr)$ over $t \in \mathbb{R}$; \Comment{ $t_{\min}$ from exact line search}
    \State set $X_{i+1} = X_{i} \exp(t_{\min}B)$;
    \EndFor       
    \Ensure $\lb Y_{\operatorname{opt}} \rb = \lb X_{\operatorname{opt}} I_{n,n_d} \rb$
  \end{algorithmic}
\end{algorithm}

\subsection{Conjugate gradient over a flag manifold}\label{sec:cg}

We present the conjugate gradient method in Algorithm~\ref{alg:cg}. Unlike steepest descent, conjugate gradient requires that we construct our new descent direction from the $(k-1)$th and
$k$th iterates, i.e., one needs to compare tangent vectors at two different points on the manifold and the only way to do this is to parallel transport the two tangent vectors to the same point. There is no avoiding parallel transport in conjugate gradient.

As the expression for parallel transport in \eqref{eq:pt2} indicates, we will need to compute
\[
e^{-\varphi_{tB}}(X) = \sum_{k=1}^{\infty} \frac{(-1)^k}{k!} \varphi_{tB}^k(X),\quad \varphi_{tB}(X) = \frac{t}{2}[B,X]_{\mathfrak{m}}.
\]
The rapid decay of the exponential series allows us to to replace it by a finite sum, reducing the task to recursively computing the iterated brackets and projection onto $\mathfrak{m}$:
\[
\varphi_{tB}^k(X) =\varphi_{tB} \circ \cdots \circ \varphi_{tB} (X)
= \Bigl(\frac{t}{2}\Bigr)^k [B, [ B,\dots,[B,X]_{\mathfrak{m}} \dots ]_{\mathfrak{m}}]_{\mathfrak{m}}.
\]
As we had pointed out at the end of Section~\ref{sec:ortho}, this step is unnecessary for the Grassmannian as $[B,X]_{\mathfrak{m}} = 0$ if $d= 1$, i.e., for $\Flag(k;n) = \Gr(k,n)$. A careful treatment of the computation of $e^{-\varphi_{tB}}(X) $ requires more details than we could go into here and is deferred to \cite{next}.

\begin{algorithm}
  \caption{Conjugate gradient in Stiefel coordinates}
  \label{alg:cg}
  \begin{algorithmic}[1]
    \Require $\lb Y_0 \rb \in \Flag(n_1,\dots,n_d;n)$ with
$Y_0\in \mathbb{R}^{n\times n_d}$ and $Y_0^\tp Y_0=I$;
    \State find $Y_0^\perp \in \mathbb{R}^{n \times (n-n_d)}$ such that $[ Y_0,Y_0^\perp] \in \O(n)$;
    \State set $X_{0} = [Y_0 ,  Y_0^\perp]$;
    \State set $G_{0} = -\nabla f(\lb Y_0 \rb)$ and $H_0 = -G_0$; \Comment{gradient at $\lb Y_{0} \rb$ as in \eqref{eq:gradient variable delta}}
    \For {$i =0,1, \dots$}
    \State compute $\widehat{B} = X_{i}^\tp  H_{i};$
    \State set $B \in \mathbb{R}^{n \times n}$ as $B_{ij} = \widehat{B}_{ij}$ for $j \leq n_d$;
    \State\phantom{set $B \in \mathbb{R}^{n \times n}$ as} $B_{ij} = -\widehat{B}_{ji}$ for $j \geq n_d$ and $i \leq n_d$;
    \State\phantom{set $B \in \mathbb{R}^{n \times n}$ as}  $B_{ij} = 0$ otherwise;
    \State minimize $ f\bigl(X_{i} \exp(tB) I_{n,n_d}\bigr)$ over $t \in \mathbb{R}$; \Comment{ $t_{\min}$ from exact line search}
    \State set $X_{i+1} = X_i \exp(t_{\min} B)$;
     \State set $Y_{i+1} = X_{i+1} I_{n,n_d}$;
    \State set $Y_{i+1}^\perp = X_{i+1} I_{n,n-n_d}$;
    \State set $G_{i+1} = -\nabla f(\lb Y_{i+1} \rb)$;
     \Comment{gradient at $\lb Y_{i+1} \rb$ as in \eqref{eq:gradient variable delta}}
      \State find $\widetilde{G}_{i+1}\in \mathbb{R}^{n \times (n-n_d)}$ such that $\widehat{G}_{i+1} = [G_{i+1},\widetilde{G}_{i+1}]$ is skew-symmetric;
    \Procedure{Descent}{$\lb Y_i \rb, \lb Y_{i+1} \rb,G_i,H_i$} \Comment{new descent direction at $\lb Y_{i+1} \rb$}
    \State $\tau H_{i} =  X_{i}  \exp(t_{\min}B) e^{-\varphi_{t_{\min}B}} ( B) I_{n,n_d}$; \Comment{parallel transport of $H_i$ as in \eqref{eq:pt2}}
    \State $\tau G_{i} = X_{i} \exp(t_{\min}B) e^{-\varphi_{t_{\min}B}} (\widehat{G}_{i} ) I_{n,n_d}$; \Comment{parallel transport of $G_i$ as in \eqref{eq:pt2}}
    \State $\gamma_i = g_{\lb Y_{i+1} \rb} (G_{i+1}-\tau G_i, G_{i+1})/g_{\lb Y_i \rb} (G_i, G_i)$;
    \Comment{$g$ as in \eqref{eq:g2}}
    \State $H_{i+1} = -G_{i+1} + \gamma_{i} \tau H_i$;
    \EndProcedure
    \State reset $H_{i+1} = -G_{i+1}$ if $i+1 \equiv 0 \bmod (k+1)(n-k)$;
    \EndFor
   \Ensure  $\lb Y_{\operatorname{opt}} \rb = \lb X_{\operatorname{opt}} I_{n,n_d} \rb$
  \end{algorithmic}
\end{algorithm}

\subsection{Newton and other algorithms over a flag manifold}

The closed-form analytic expressions derived in this article permit one to readily extend other optimization algorithms on Euclidean spaces to flag manifolds. For example, the \emph{Newton search direction} is given by the  tangent vector $X \in T_{\lb Y \rb} \Flag(n_1,\dots,n_d;n)$ such that 
\[
\nabla^2f(\lb Y \rb) (X,X') = g_{\lb Y \rb}\bigl(-\nabla f(\lb Y \rb),X' \bigr) = -\tr (f_Y^\tp  X'),
\]
for every $X'\in T_{\lb Y \rb} \Flag(n_1,\dots,n_d;n)$, which gives us a system of linear equations upon plugging in the expression for Riemannian Hessian in \eqref{eq:Hessian}. Using the Newton search direction for $G_i$ in Algorithm~\ref{alg:sd} then gives us Newton method on the flag manifold. In a similar vein,  one may derive other standard algorithms for unconstrained optimization, e.g., quasi-Newton method, accelerated gradient descent, stochastic gradient descent, trust region methods, etc, for the flag manifold. Nevertheless, given that the goal of our article is to develop foundational material, we will leave these  to  future work \cite{next}.

\section{Numerical experiments}\label{sec:num}

We will test our algorithm for steepest descent on the flag manifold numerically. As we explained in Section~\ref{sec:cg}, the experiments for conjugate gradient algorithm is more involved and is deferred to \cite{next}. We run our numerical experiments on two problems: (i) the  principal flag problem in Section~\ref{sec:dis} is one for which the solution may be determined in closed-form analytically, and thus it serves to  demonstrate the correctness of our algorithm, i.e., converges to the true solution; (ii) a variation of the previous problem with a more complicated objective function to show that the convergence behavior remains unchanged. In addition, neither problem can be solved by simply treating them as nonlinear optimization problems with equality constraints and applying standard nonlinear optimization algorithms.

In the following we will assume the identification in \eqref{eq:identify} and  use Stiefel coordinates throughout.

\subsection{Principal flags}\label{sec:dis}

Let $M \in \mathbb{R}^{n \times n}$ be symmetric.
We seek the solution to
\begin{equation}\label{eq:dis}
\begin{tabular}{rl}
maximize & $ \tr(Y^\tp MY)$ \\[1ex]
subject to & $\lb Y \rb \in \Flag(n_1,\dots,n_d; n)$.
\end{tabular}
\end{equation}
Here $Y \in \mathbb{R}^{n \times n_d}$, $Y^\tp Y = 1$, and the objective function is well-defined as a function on the flag manifold: If we have $Y$ and $Y'$ with $\lb Y \rb = \lb Y'\rb$, then they must be related as in \eqref{eq:rep} and thus $ \tr(Y^\tp MY) =  \tr(Y^{\prime\tp} MY')$.

As we saw in Example~\ref{eg:pca}, when $M$ is a sample covariance matrix, the solution to \eqref{eq:dis} is equivalent to PCA when we seek a complete flag, i.e., $d=n-1$ and $n_i =i$, $i =1,\dots,n-1$. An advantage proffered by this approach is that if we do not know the intrinsic dimension of the data set a priori, then finding the flag as opposed to any particular subspace gives us the entire profile, showing how increasing dimension accounts for an increasingly amount of variance. The problem in \eqref{eq:dis} is thus a generalization of PCA, allowing us to seek any flag, not necessarily a complete one. It may also be interpreted as finding subspaces of dimensions $n_1,n_2-n_1,\dots,n_d-n_{d-1}$ that are independent and explain different levels of variance in the data set.

Figure~\ref{fig:convergence} shows the convergence trajectories of steepest descent, i.e., Algorithm~\ref{alg:sd}, on the flag manifold $\Flag(3,7,12; 60)$, a $623$-dimensional manifold. The symmetric matrix $M \in \mathbb{R}^{60 \times 60}$ is generated randomly with standard normal entries. Since the true solution of \eqref{eq:dis} may be determined in closed form --- it is the sum of the $k$ largest eigenvalues of $M$ --- we may therefore conclude that Algorithm~\ref{alg:sd}  converges to the true solution in around $80$ iterations. Indeed the function values stabilize after as few as $10$ iterations. At least for this problem, we see that the vanishing of the Riemannian gradient serves as a viable stopping condition. In our implementation, our stopping conditions are determined by (i) Frobenius norm of Riemannian gradient, (ii) distance between successive iterates, and (iii) number of iterations.

\begin{figure}[h]
  \centering
    \includegraphics[trim=6ex 11ex 6ex 11ex, clip, scale=0.6]{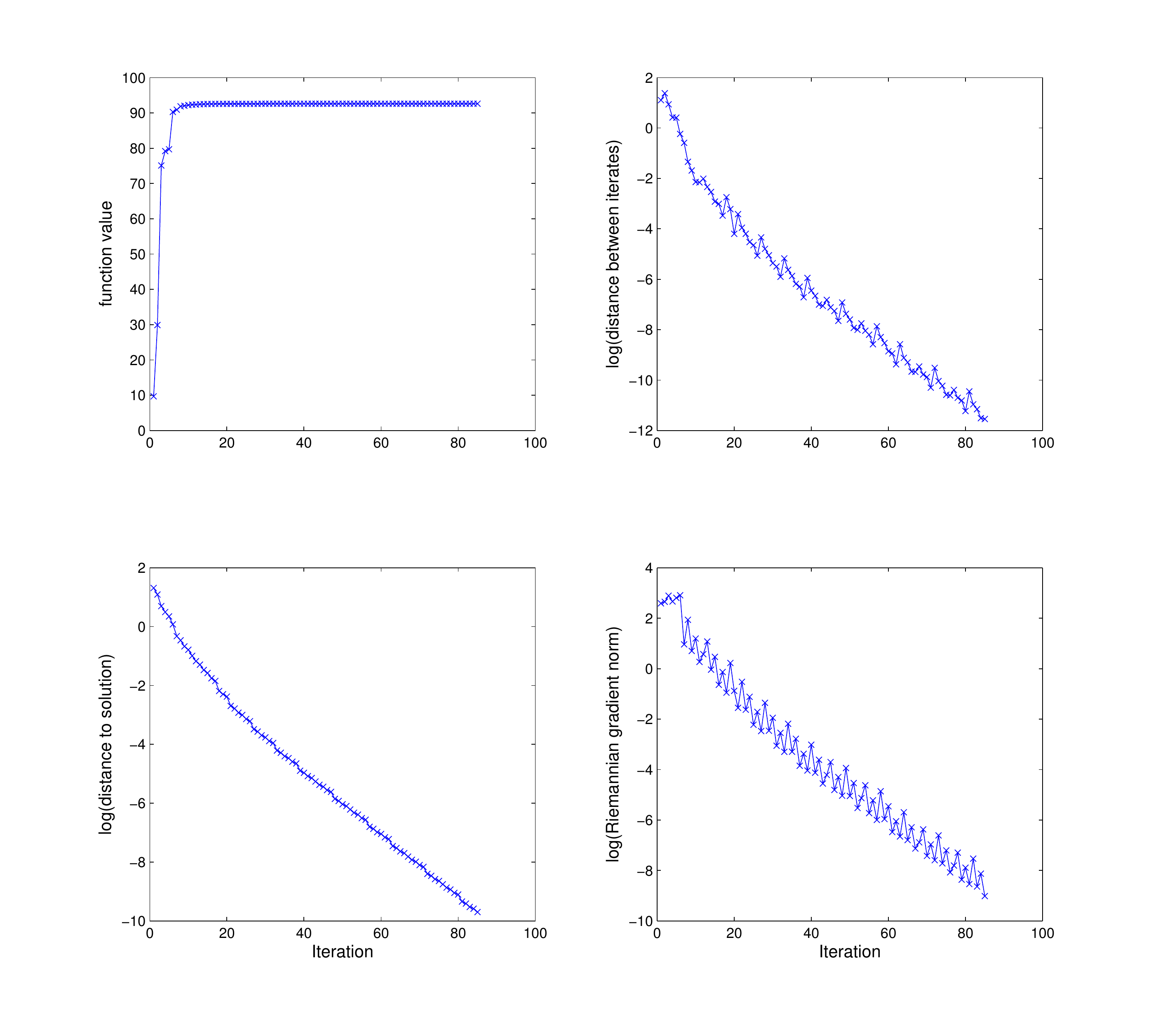}
\vspace*{-1ex}
      \caption{Convergence trajectories  for \eqref{eq:dis} on $\Flag(3,7,12; 60)$.}
    \label{fig:convergence}
\end{figure}

We perform extensive experiments beyond that in Figure~\ref{fig:convergence} by taking average of $100$ instances of the problem \eqref{eq:dis} for various values of $n_1,\dots,n_d$. We tabulate our results showing accuracy and speed in Tables~\ref{table:distdis1}--\ref{table:timedis2}. Tables~\ref{table:distdis1} and \ref{table:distdis2} show that Algorithm~\ref{alg:sd} is robust across all dimensions of flags and  ambient spaces that we tested. Tables~\ref{table:timedis1} and \ref{table:timedis2} show that elapsed time taken for Algorithm~\ref{alg:sd} increases roughly linearly with the dimension of the flag manifold.

\begin{table}[ht]
\renewcommand{\arraystretch}{1.3}
	\centering
	\begin{tabular}{r|rrrrrrrrrr}
		$k$ & 30 & 40 & 50 & 60 & 70 & 80 & 90 & 100\\
		\hline\hline
		Accuracy ($\times 10^{-4}$) & $2$ & $8$ & $64$ & $32$ & $4$ & $87$ & $20$ & $15$ \\ 
	\end{tabular} 
\vspace{1ex}
	\caption{Distance to true solution for \eqref{eq:dis} on  $\Flag(3,9,21;k)$.}
	\label{table:distdis1}
\vspace{-4ex}
\end{table}

\begin{table}[ht]
\renewcommand{\arraystretch}{1.3}
	\centering
	\begin{tabular}{r|rrrrrrrrrr}
		$k$ & 30 & 40 & 50 & 60 & 70 & 80 & 90 & 100\\
		\hline\hline
		Elapsed Time & $0.38$ & $0.40$ & $0.67$ & $0.93$ & $1.71$ & $2.27$ & $3.08$ & $4.07$ \\ 
	\end{tabular} 
\vspace{1ex}
	\caption{Elapsed time for \eqref{eq:dis} on $\Flag(3,9,21;k)$.}
	\label{table:timedis1}
\vspace{-4ex}
\end{table}

\begin{table}[ht]
\renewcommand{\arraystretch}{1.3}
	\centering
	\begin{tabular}{r|rrrrrrrrrrrr}
		$k$ & 1 & 2 & 3 & 4 & 5 & 6 & 7 & 8 & 9 & 10\\
		\hline\hline
		Accuracy ($\times 10^{-4}$) & $1.4$ & $3.4$ & $3.4$ & $8.6$ & $2.8$ & $18$ & $19$ & $5.1$ & $9.3$ & $11$\\ 
	\end{tabular} 
\vspace{1ex}
	\caption{Distance to true solution for \eqref{eq:dis} on $\Flag(2,\dots,2k;60)$.}
	\label{table:distdis2}
\vspace{-4ex}
\end{table}

\begin{table}[ht]
\renewcommand{\arraystretch}{1.3}
	\centering
	\begin{tabular}{r|rrrrrrrrrrrr}
		$k$ & 1 & 2 & 3 & 4 & 5 & 6 & 7 & 8 & 9 & 10\\
		\hline\hline
		Elapsed Time & $0.54$ & $0.81$ & $0.79$ & $0.96$ & $1.05$ & $0.91$ & $1.20$ & $1.06$ & $1.18$ & $1.12$\\ 
	\end{tabular} 
\vspace{1ex}
	\caption{Elapsed time for \eqref{eq:dis} on $\Flag(2,\dots,2k;60)$.}
	\label{table:timedis2}
\vspace{-4ex}
\end{table}

\begin{table}[ht]
\renewcommand{\arraystretch}{1.3}
	\centering
	\begin{tabular}{r|rrrrrrrrrrrr}
		$k$ & 1 & 2 & 3 & 4 & 5 & 6 & 7 & 8 & 9 & 10\\
		\hline\hline
		Accuracy ($\times 10^{-4}$) & $1.4$ & $3.4$ & $3.4$ & $8.6$ & $2.8$ & $18$ & $19$ & $5.1$ & $9.3$ & $11$\\ 
	\end{tabular} 
\vspace{1ex}
	\caption{Distance to true solution for \eqref{eq:dis} on $\Flag(2,\dots,2k;60)$.}
	\label{table:distdis3}
\vspace{-4ex}
\end{table}

\subsection{Nonlinear eigenflags}\label{sec:disv}

This is a variation of the principal flag problem \eqref{eq:dis}:
\begin{equation}\label{eq:eig}
\begin{tabular}{rl}
maximize & $\sum_{i=1}^{d}\tr(Y_i^\tp MY_i)^2$\\[1ex]
subject to & $\lb Y_1,\dots,Y_d \rb \in \Flag(n_1,\dots,n_d; n)$.
\end{tabular}
\end{equation}
Again $M \in \mathbb{R}^{n \times n}$ is a symmetric matrix and the flag is given in Stiefel coordinates $Y = [Y_1,\dots,Y_d] \in \mathbb{R}^{n \times n_d}$, $Y^\tp Y = I$, but partitioned into submatrices $Y_i \in \mathbb{R}^{n \times (n_i - n_{i-1})}$, $Y_i^\tp Y_i = I$, $i =1,\dots,d$. More generally, the objective function in  \eqref{eq:eig} may be replaced by $\sum_{i=1}^{d}f_i\bigl(\tr(Y_i^\tp M Y_i)\bigr)$ with $f_1,\dots,f_d \in C^2(\mathbb{R})$. Choosing $f_1(x) = \dots = f_d(x) = x$ gives us \eqref{eq:dis} and choosing $f_1(x) = \dots = f_d(x) = x^2$ gives us \eqref{eq:eig}. Note that it will take considerable effort to formulate a  problem like \eqref{eq:eig} as a constrained optimization problem in Euclidean space.
\begin{figure}[h]
  \centering
    \includegraphics[trim=6ex 11ex 6ex 10ex, clip, scale=0.6]{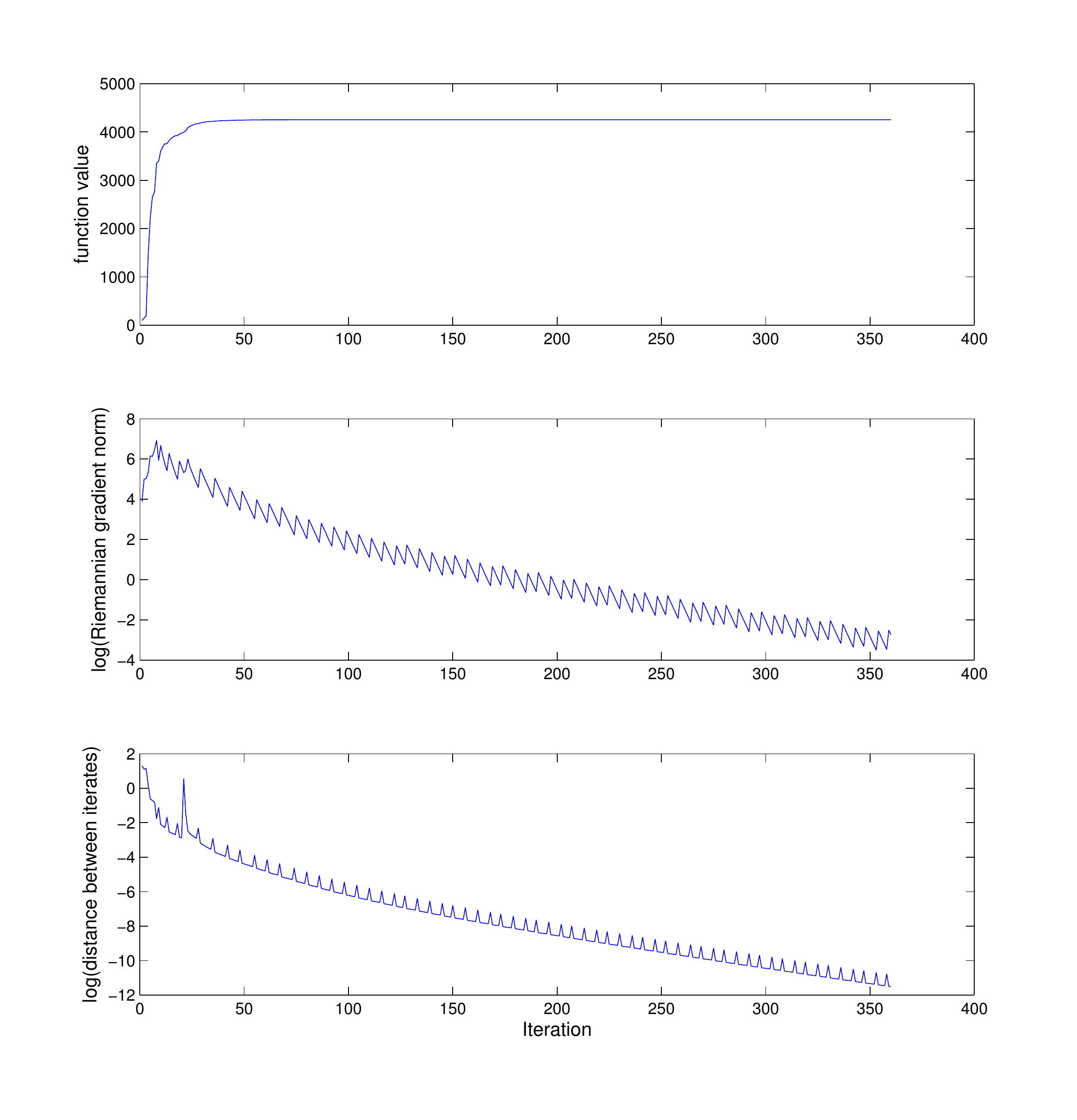}
      \caption{Convergence trajectories for \eqref{eq:eig} on $\Flag(3,7,12; 60)$.}
\vspace*{-2ex}
    \label{fig:convergence2}
\end{figure}

The convergence trajectories for Algorithm~\ref{alg:sd} applied to \eqref{eq:eig} are shown in Figure~\ref{fig:convergence2}. The nonlinearity imposes a cost --- it takes around $390$ iterations to satisfy one of the our stopping criteria, although the function values stabilize after around $60$ iterations. The jagged spikes seen in Figure~\ref{fig:convergence2} are a result of  iterates moving along a geodesic and then jumping to another geodesic. So this is indicative of steepest descent following a path that comprises multiple geodesics. A caveat is that unlike the principal flag problem \eqref{eq:dis}, we do not have a closed-form solution for \eqref{eq:eig} and thus we may only guarantee convergence to a local minimizer, which is reflected in Figure~\ref{fig:convergence2}.

\section{Conclusion}

For most of its history, continuous optimization has been concerned with optimizing functions over the Euclidean space $\mathbb{R}^n$; but this has begun to change with the advent of semidefinite programming \cite{sdp} and orthogonality-constrained optimization \cite{EAS}, where objective functions are naturally defined over the positive definite cone $\mathbb{S}^n_{\pp}$, the Stiefel manifold $\V(k,n)$, and the Grassmannian $\Gr(k,n)$. These developments have provided us with the capacity to optimize over not just vectors but also covariances matrices, orthonormal bases, and subspaces.  The work here extends such capabilities to flags, which capture nested structures  in multilevel, multiresolution, or multiscale phenomena.
In future works, we will investigate computational issues \cite{next} that have been deferred from this first study. We will also examine \emph{complex} flag manifolds, i.e., where the vector spaces involved are over $\mathbb{C}$. Its properties will be quite different --- for example, we saw in Proposition~\ref{prop:geometric structure} that a real flag manifold is both an affine and a projective variety; but the only complex algebraic varieties that are both projective and affine are finite sets. So complex flag manifolds will lack some of the properties discussed here, although it will have others, e.g., a symplectic structure.

\bibliographystyle{abbrv}

\end{document}